\documentclass[11pt,leqno]{amsart}
\usepackage{graphicx} %
\usepackage[margin=1.1in]{geometry}

\usepackage{color} 

\usepackage{bm} 
\usepackage{comment}

\usepackage{mathtools}
\usepackage{extarrows}
\usepackage{caption}
\usepackage{subcaption}

\usepackage{enumerate}
\usepackage{mathabx}

\mathtoolsset{showonlyrefs=true,showmanualtags=true}

\newcommand{\veps}{\varepsilon}

\newcommand{\Var}{\mathrm{Var}}
\DeclareMathOperator{\supp}{supp}

\newtheorem{theorem}{Theorem}[section]
 
\newtheorem{corollary}[theorem]{Corollary}
\newtheorem{definition}[theorem]{Definition}
\newtheorem{example}[theorem]{Example}
\newtheorem{lemma}[theorem]{Lemma}

\newtheorem{proposition}[theorem]{Proposition}
\newtheorem{assumption}[theorem]{Assumption} 
\theoremstyle{remark}
\newtheorem{remark}[theorem]{Remark}

\newcommand{\R}{\mathbb R}
\renewcommand{\epsilon}{\varepsilon}

\definecolor{cherry}{rgb}{0.8,0.1,0.5}

\newcommand{\E}{\mathbb{E}}

\DeclarePairedDelimiter{\abs}{\lvert}{\rvert}
\DeclarePairedDelimiter{\norm}{\lVert}{\rVert}

\DeclarePairedDelimiterX\paren[1]{(}{)}{%
#1%
}
\DeclarePairedDelimiterX\brak[1]{[}{]}{%
  #1%
}

\DeclarePairedDelimiterX\set[1]\{\}{%
  #1%
}

\newcounter{GICase}
\newcommand{\restartcases}{\setcounter{GICase}{0}}

\newcommand{\case}[1][\relax]{\smallskip\par\noindent\stepcounter{GICase}\emph{Case \Roman{GICase}: \ifx#1\relax\relax\else#1.\fi}}

\newcounter{GIStep}
\newcommand{\restartsteps}[1][0]{\setcounter{GIStep}{#1}}

\newcommand{\step}[1][\relax]{\smallskip\par\noindent\stepcounter{GIStep}\emph{Step \arabic{GIStep}: \ifx#1\relax\relax\else#1.\fi}}

\newcommand{\var}{\mathrm{Var}}
\newcommand{\osc}{\mathrm{osc}}
\usepackage{bbm}
\usepackage{mathrsfs}

\usepackage{float}
\usepackage{subfig}

\title{Sample complexity for divergence regularized optimal transport with radial cost}

\author{Ruiyu Han}
\address[Ruiyu Han]
{\newline
\mbox{}\hspace{0.3cm} Department of Mathematics\newline
\mbox{}\hspace{0.3cm} Carnegie Mellon University\newline
\mbox{}\hspace{0.3cm} Pittsburgh, PA 15213}
\email{ruiyuh@andrew.cmu.edu}

\author{Johannes Wiesel}

\address[Johannes Wiesel] 
{\newline 
\mbox{}\hspace{0.3cm} Department of Mathematics\newline
\mbox{}\hspace{0.3cm} University of Copenhagen\newline
\mbox{}\hspace{0.3cm} Universitetsparken 5\newline
\mbox{}\hspace{0.3cm} 2100 Copenhagen, Denmark}
\email{wiesel@math.ku.dk}

\date{\today}

\begin{document}

\maketitle

\begin{abstract}
We prove a new sample complexity result for divergence regularized optimal transport. Our bound holds for probability measures on~$\R^d$ with exponential tail decay and for radial cost functions that satisfy a local Lipschitz condition. It is sharp up to logarithmic factors, and captures the intrinsic dimension of the marginal distributions through a generalized covering number of their supports.  Examples that fit into our framework include subexponential and subgaussian distributions and radial cost functions~$c(x,y)=|x-y|^p$ for~$p\ge 1$ with logarithmic entropy or polynomial~$\alpha$-divergence.
\end{abstract}

\section{Introduction}

Let~$\mu,\nu$ be probability measures on~$\R^d$ for some~$d\ge 1$, and assume that we are given i.i.d samples~$X_1,...,X_n$,~$Y_1,...,Y_n$ drawn from~$\mu$ and~$\nu$ respectively. Define the empirical measures
\begin{equation*}
    \mu_n := \frac{1}{n}\sum_{i=1}^{n}\delta_{X_i}, \quad\nu_n := \frac{1}{n}\sum_{i=1}^{n}\delta_{Y_i}.
\end{equation*}
Many works in statistical optimal transport have studied comparisons of the optimal transport problem 
\begin{align}\label{eq: OT}
\mathcal{C}_0(\mu,\nu):=\inf\limits_{\pi\in\Pi(\mu,\nu)}\int c(x,y)\,\pi(dx,dy) \tag{OT}
\end{align}
with its empirical counterpart~$\mathcal{C}_0(\mu_n,\nu_n)$. Here ~$c:\mathbb R^d\times\mathbb R^d\to\mathbb R$ is a cost function,~$\Pi(\mu,\nu)$ denotes the set of probability distributions~$\pi$ on~$\R^d\times \R^d$ with marginals~$\mu$ and~$\nu$ and~$\mathcal{C}_0(\mu_n,\nu_n)$ is defined as in~\eqref{eq: OT} with the empirical measures~$\mu_n, \nu_n$ replacing their population versions~$\mu,\nu$; we refer to \cite{villani2009optimal, santambrogio2015optimal} for fundamental properties of~\eqref{eq: OT}. It is well-known that comparisons between~$\mathcal{C}_0(\mu,\nu)$ and~$\mathcal{C}_0(\mu_n,\nu_n)$ suffer from the so-called \emph{curse of dimensionality}, i.e.~the difference~$\E[|\mathcal{C}_0(\mu_n\,\nu_n)-\mathcal{C}_0(\mu,\nu)|]$ scales like~$n^{-1/d}$ in general; see \cite{dudley1969speed, fournier2015rate, weed2019sharp}. This severely restricts applications of OT to high-dimensional data sets. The most popular remedy for this issue is to add a penalization term to~\eqref{eq: OT}: for~$\epsilon>0$, the \emph{regularized optimal transport (ROT)} problem is given by
\begin{equation}\label{eq: ROT}
    \mathcal{C}_{\veps}(\mu,\nu) := \inf\limits_{\pi\in\Pi(\mu,\nu)}\int c\,d\pi+\veps D_{\varphi}(\pi|\mu\otimes\nu).\tag{ROT}
\end{equation}
Here~$\mu\otimes \nu$ is the product coupling of~$\mu$ and~$\nu$, and~$D_{\varphi}$ is a~$\varphi$-divergence, defined as
\begin{equation}
 D_{\varphi}(\pi|\rho) = \begin{cases}
    \int \varphi\Big(\frac{d\pi}{d\rho}\Big) d\rho & \pi\ll \rho,\\
    \infty &\text{otherwise}
\end{cases}
\end{equation}
for~$\pi, \rho\in\mathcal P(\R^d\times \R^d)$ and~$\varphi:[0,\infty) \to\R$. The most famous examples of ROT are \emph{entropic optimal transport (EOT)}, where~$\varphi(x)=x\log(x)$ and \emph{quadratic optimal transport (QOT)}, where~$\varphi(x)=|x|^2/2$. Introduced in \cite{galichon_saliane_2010_matching, cuturi2013sinkhorn} to speed up computation of OT problems, EOT has by now undergone an extensive investigation in the statistical sciences (see Section~\ref{sec:related} below for an overview of recent advances). On the contrary, quadratic regularization has become popular in the mathematical sciences only quite recently. Early work on QOT includes \cite{muzellec2017tsallis, blondel2018smooth, essid2018quadratically}, and highlights its superior approximation properties of OT maps, see e.g. \cite{gonzalez2024sparsity, wiesel2025sparsity}, as well as its numerical stability for small penalization parameter~$\epsilon,$ compared to EOT.

In this paper, we aim to find upper bounds for the quantity 
\begin{equation} \label{eq:aim}
   \E \left[\left|\mathcal{C}_{\veps}(\mu_n,\nu_n)-\mathcal{C}_{\veps}(\mu,\nu)\right|\right].
\end{equation}

The problem of bounding~\eqref{eq:aim} for EOT goes back at least to \cite{genevay2019sample, chizat2020faster}. The case of subgaussian measures~$\mu,\nu$ with quadratic cost has been addressed in \cite{mena2019statistical}. More recently, \cite{rigollet2022sample} derives dimension-free bounds for bounded cost functions. However, the rates scale exponentially in~$1/\epsilon$. Our method and setting is most closely related to the subsequent work \cite{stromme2023minimum}, that assumes continuous cost functions on compact spaces.  Let us also mention \cite{bayraktar2025stability}, that derive non-optimal rates for~\eqref{eq:aim} and OT problems regularized by general divergences. To the best of our knowledge, our article is the first work to derive sharp bounds on the sample complexity~\eqref{eq:aim} for radial (unbounded) cost functions on unbounded spaces. Our main result, Theorem~\ref{thm: main}, states that under fairly general assumptions on~$c$ and~$\mu,\nu$, the quantity~\eqref{eq:aim} is of order~$1/\sqrt{n}$ up to logarithmic factors. We achieve this by extending the methodology of \cite{stromme2023minimum} to probability measures with exponential tail decay. As in Stromme's work, our rates depend on the minimum of the covering numbers of the (appropriately normalized) supports of~$\mu,\nu$ which is a concept called \emph{minimum intrinsic dimension scaling} of EOT. We provide a more detailed comparison of Theorem~\ref{thm: main} with the works mentioned above in Section~\ref{sec:discussion}.

While our methodology offers sharp estimates for the sample complexity of the ROT cost in a quite general setting, our approach does not directly extend to sample complexity estimates of dual potentials resp.~OT transport maps.%
\subsection{Related work}\label{sec:related}
The literature on statistical OT has grown tremendously in the last couple of years. Instead of providing a complete literature review, we refer to \cite{panaretos2020invitation, chewi2025statistical} for an overview and only highlight a few landmark papers here.

OT has found many applications in statistics recently, see \cite{carlier2016vector, chernozhukov2017monge, hallin2021distribution, ghosal2022multivariate, wiesel2022measuring} and the references therein.

As mentioned above, determining the sample complexity for OT problems has a long history; see \cite{fournier2015rate} and the references therein. Recently \cite{hundrieser2024empirical} shows that, similar to the EOT case discussed here, the convergence of the empirical OT problem is determined by the less complex marginal law.%

Turning to EOT, apart from the sample complexity results mentioned above, significant progress has also been made in finding distributional limits for empirical entropic optimal transport quantities, see \cite{gonzalez2022weak, goldfeld2024statistical, goldfeld2024limit, mordant2024entropic, gonzalez2023weak, del2023improved} and the references therein. We also remark that convergence of EOT to the OT problem for~$\epsilon\to 0$ is of independent interest and has been studied e.g.~in \cite{pal2019difference, chizat2020faster, conforti2021formula, pooladian2021entropic,altschuler2022asymptotics, nutz2022entropic}. Lastly, apart from its superior sample complexity, EOT also offers better computational complexity as observed e.g.~in \cite{altschuler2017near}.

Next to the initial works sparking research in QOT mentioned in the Introduction, recent advances include \cite{lorenz2021quadratically, nutz2025quadratically, gonzalez2025linear}. Let us also emphasize that \cite{gonzalez2025sample} derives (parametric) central limit theorems for dual potentials, optimal couplings, and optimal costs. Our main result applied to QOT provides the corresponding finite sample guarantees. 

\subsection{Notation}

We equip~$\R^d$ with the Euclidean norm~$|\cdot|$ and denote the open ball of radius~$r>0$ around the point~$x\in \R^d$ by~$B_r(x)$. We write~$B_r:=B_r(0)$ for simplicity. We denote the complement of a set~$A\subseteq \R^d$ by ~$A^c$. 
The set of (Borel) probability measures on~$\R^d$ is denoted by~$\mathcal{P}(\R^d).$ If~$\mu \in \mathcal{P}(\R^d)$ and~$A\subseteq \R^d$ is a Borel set, then~$\mu|_A(\cdot ):=\mu(\cdot \cap A)$ is the restriction of~$\mu$ to~$A.$ We denote the product measure of two probability measures~$\mu,\nu\in \mathcal{P}(\R^d)$ by~$\mu\otimes \nu.$ We write~$M_p( \mu) :=(\int\|x\|^p\,  \mu(dx))^{1/p}$ for~$ \mu\in\mathcal{P}(\mathbb{R}^d)$ and write~$\mathrm{spt}(\mu)$ for the support of~$\mu$. Using the same notation as in~\eqref{eq: OT} above, we define the~$p$-Wasserstein distance on~$\mathcal{P}(\R^d)$ as 
\begin{align*}
W_p(\mu,\nu)^p :=\inf_{\pi\in \Pi(\mu,\nu)} \int |x-y|^p \,\pi(dx,dy).  
\end{align*}
For measures~$\pi, \tilde{\pi}\in \mathcal{P}(\R^d\times \R^d)$ we  define the~$p$-Wasserstein distance 
\begin{equation}\label{eq:defwp}
W_p(\pi,\tilde \pi)^p :=\inf_{\gamma\in \Pi(\pi,\tilde \pi)} \int [|x_1-y_1|^p+|x_2-y_2|^p ]\,\gamma(dx,dy)
\end{equation}
for~$x=(x_1,x_2), y=(y_1,y_2)\in \R^{d\times d}$, in accordance with \cite{eckstein2022quantitativestabilityregularizedoptimal}.
The covering number of a set~$A\subseteq \mathbb R^d$ at scale~$\delta>0$ is defined as 
\begin{equation}
    \mathcal{N}(A,\delta):=\min\Big\{k\in\mathbb N\:|\: \exists x_1,...,x_k\in \mathbb R^d\::\: A\subseteq \bigcup_{\ell=1}^{k} B_\delta(x_\ell)\Big\}.
\end{equation}
The incomplete Gamma function is given by
\begin{align}\label{eq:gamma}
\Gamma(s,x) := \int_{x}^{\infty}t^{s-1}e^{-t}\, d t, 
\end{align}
where~$x\ge 0$ and~$s>0$,
and the Gamma function is~$\Gamma(s):=\Gamma(s,0).$ 
We denote constants by~$C$, with the convention that~$C$ can increase from line to line. We always state the dependence of constants on quantities of interest explicitly.

\section{Main Result}

Throughout the paper we make three assumptions. The first one states that the tails of~$\mu,\nu$ decay exponentially.

\begin{assumption}\label{assumption:Psialpha}
 There exist constants~$c_{\mu}, c_{\nu}>0$ and~$ \alpha_{\mu},\alpha_{\nu}\geq 1$
such that
\begin{equation}\label{eq:PsiAlapha}
\mu(B^c_{r})\leq 2\exp\big(-c_{\mu}r^{\alpha_{\mu}}\big),\qquad \nu(B^c_{s})\leq 2\exp\big(-c_{\nu}s^{\alpha_{\nu}}\big)
\end{equation}
for all~$r,s>0$.
\end{assumption}
\noindent Well-known distributions satisfying Assumption~\ref{assumption:Psialpha} are
subgaussian distributions ($\alpha_\mu=\alpha_\nu=2$), subexponential distributions ($\alpha_\mu=\alpha_\nu=1$) or more generally, probability measures on Orlicz spaces of exponential type. 

We also make an assumption on the shape of the cost function~$c$.

\begin{assumption}\label{assumption: cost function}
The cost function satisfies~$c(x,y)=h(|x-y|)$ for some continuous function~$h:\mathbb [0,\infty) \to [0,\infty)$ with~$h(0)=0$, and there exist constants~$p\geq 1$ and~$C_p>0$ such that
\begin{equation}\label{eq:devh}
   |h(t)-h(t')|\le C_p (t\vee t')^{p-1} |t-t'|,\quad \forall  t, t'>0.
\end{equation}
\end{assumption}
\noindent Important examples of cost functions satisfying Assumption~\ref{assumption: cost function} are~$c(x,y)=|x-y|^p$ for~$p\ge 1.$ Lastly we make an assumption on the divergence~$\varphi$.

\begin{assumption}\label{assum:psi}
 The function~$\varphi: [0,\infty)\to \R$ is strictly convex with
$\varphi(1)=0$,~$\lim_{x\to\infty}\varphi(x)/x=+\infty$ and such that the convex conjugate
\begin{equation}
 y\mapsto \psi(y):= \varphi^{*}(y):= \sup_{x\geq 0}\{xy-\varphi(x)\}
\end{equation}
is in~$C^1(\mathbb{R})$. Moreover, there exists~$t_0>0$ and~$\delta\in (0, t_0)$ such that~$\psi'(t_0)=1$ and~$\psi$ is strictly convex and~$C^2$ on~$[t_0-\delta,+\infty)$. In addition, there exists~$C_{\psi}>0, \gamma>0$ such that for all~$x\in [t_0-\delta,+\infty)$
\begin{equation}\label{eq:psi}
    \abs[\Big]{\frac{\psi''(x)}{\psi'(x)}}\leq \frac{C_{\psi}}{|x|^{\gamma}}, \quad \gamma\geq 0.
\end{equation}
\end{assumption}

Note that Assumption~\ref{assum:psi} is similar to \cite[Assumption 2.1]{gonzalez2025sparse}. Compared to their assumption, we however do not require~$\psi$ to be~$C^2(\R)$, so that QOT is included in our setting. 
\begin{example}
For EOT we have~$\varphi(x)=x\log x$, which yields~$\psi'(y)=e^{y-1}$ and~$\psi''(y)=e^{y-1}$, so that Assumption~\ref{assum:psi} is satisfied with~$t_0=1, C_{\psi}=1, \gamma=0$ and arbitrary~$\delta\in (0,1)$. Next, consider the polynomial (Tsallis) divergence,
\begin{equation}
    \varphi(x)=\frac{x^{\alpha}-1}{\alpha}
\end{equation}
for~$\alpha\in (1,\infty)$, where the case~$p=2$ corresponds to QOT. Then~$\psi'(y)=y_{+}^{\beta-1}$ where~$1/\alpha+1/\beta=1$ and~$y_{+}:=\max\{y,0\}$. For any~$\alpha\in (1,\infty)$, it can be checked that~$\varphi$ satisfies Assumption~\ref{assum:psi} with ~$t_0=1, C_{\psi}=\beta-1, \gamma=1$ and arbitrary~$\delta\in (0,1)$.
\end{example}

We are now in a position to state our main result.

\begin{theorem}\label{thm: main}
Let Assumptions~\ref{assumption:Psialpha}-\ref{assum:psi} hold. We define
\begin{align}\label{eq:chooseRxRy}
\begin{split}
 r^\mu_n &:= \Big[ 4pc_{\mu}^{-1}\paren{c_{\mu}^{-1}\vee 1}\Big(\frac{p}{\alpha_{\mu}}\vee 1\Big)^2\log(n)\Big]^{\frac{1}{\alpha_{\mu}}},\\
 r^\nu_n &:= \Big[4pc_{\nu}^{-1}\paren{c_{\nu}^{-1}\vee 1}\Big(\frac{p}{\alpha_{\nu}}\vee 1\Big)^2\log(n)\Big]^{\frac{1}{\alpha_{\nu}}},
 \end{split}
 \end{align}
and 
\begin{align}\label{eq:Bn}
B_n^\mu := B_{r_n^\mu}(0)\cap \mathrm{spt}(\mu),\qquad B_n^\nu  := B_{r_n^\nu}(0)\cap \mathrm{spt}(\nu).
\end{align}
Then
\begin{align}
&\E [|\mathcal{C}_{\veps}(\mu_n,\nu_n)-\mathcal{C}_{\veps}(\mu,\nu)|]
\leq \frac{C}{\sqrt{n}}\Big(1+c_{\mu}^{-\frac{p}{\alpha_{\mu}}}+c_{\nu}^{-\frac{p}{\alpha_{\nu}}}\Big)+\frac{C}{\sqrt{n}} \big( \veps+(r_n^\mu+r_n^\nu)^{p}\big)\\
&\hspace{4.5cm}\cdot 
\paren[\bigg]{1+e^{\frac{C_{\psi}\tilde{\delta}}{2(t_0/2)^{\gamma}}}\sqrt{\mathcal{N} \Big(B^\mu_n, \frac{\tilde{\delta}\veps}{2C_p(r^\mu_n+r^\nu_n)^{p-1}}\Big)\wedge \mathcal{N}\Big(B^\nu_n, \frac{\tilde{\delta}\veps}{2C_p(r^\mu_n+r^\nu_n)^{p-1}}\Big)}} \label{eq:main}
\end{align}
holds for all~$n\ge 5$, where the constant~$C$ only depends on~$\alpha_{\mu},\alpha_{\nu},p, C_p, t_0$ and 
\begin{align}\label{eq:delta}
     \tilde{\delta}=\min\{\delta, t_0/2\}
\end{align}
\end{theorem}

The main idea behind the proof of Theorem~\ref{thm: main} relies on a careful approximation of the difference~$\E[|\mathcal{C}_{\veps}(\mu_n,\nu_n)-\mathcal{C}_{\veps}(\mu,\nu)|]$ with probability measures that are supported on the closures of~$ B_r,  B_s$ for appropriately chosen~$r,s>0$. More concretely, let~$\mu^{B_r}$ be the conditional distribution of~$\mu$ given~$\{x\in B_r\}$ and define~$\nu^{B_s}$ similarly. We then write
\begin{align}\label{eq:main_idea_proof}
\begin{split}
\E[|\mathcal{C}_{\veps}(\mu_n,\nu_n)-\mathcal{C}_{\veps}(\mu,\nu)|] &\le  \E[|\mathcal{C}_{\veps}(\mu_n,\nu_n)-\mathcal{C}_{\veps}(\mu^{B_r}_n,\nu^{B_s}_n)|]\\
&\quad +\E[|\mathcal{C}_{\veps}(\mu^{B_r}_n,\nu^{B_s}_n)-\mathcal{C}_{\veps}(\mu^{B_r},\nu^{B_s})|]\\
&\quad +|\mathcal{C}_{\veps}(\mu^{B_r},\nu^{B_s})-\mathcal{C}_{\veps}(\mu,\nu)|   
\end{split}
\end{align}
and estimate the three summands on the right-hand side of~\eqref{eq:main_idea_proof} separately. Compared to existing results in the literature, this allows us to derive bounds, that depend on~$c$ only through Assumption~\ref{assumption: cost function}. In particular, our results do not rely on structural assumptions or  smoothness of the cost function, nor on smoothness of the dual potentials.

\begin{remark}
Let us state briefly how the constant~$C$ in Theorem~\ref{thm: main} depends on~$\alpha_{\mu},\alpha_{\nu}$ and~$p, C_p, t_0$: we have polynomial dependence on~$C_p$ and~$t_0$. The dependence on~$p$ is exponential. Moreover,~$C$ depends exponentially on~$1/\alpha_{\mu}$,~$1/\alpha_{\nu}$, as it contains a multiplicative factor of~$\Gamma(p/\alpha_{\mu})$ and~$\Gamma(p/\alpha_{\nu})$. This aligns with our intuition that the larger~$\alpha_{\mu},\alpha_{\nu}$, the more quickly the tails decay, and the closer we are to the compactly supported setting.
\end{remark}

\textbf{Structure of the article}.
The remainder of this article is structured as follows: we give examples of Theorem~\ref{thm: main} in Section~\ref{sec:discussion}. Section~\ref{sec:preliminary} collects some preliminary results needed for the proof of Theorem~\ref{thm: main}. The first and last terms in~\eqref{eq:main_idea_proof} are estimated in  Section~\ref{sec:compact}, using results from \cite{eckstein2022quantitativestabilityregularizedoptimal}, while the middle term is estimated in Section~\ref{sec:stromme} using results from \cite{stromme2023minimum}.  We state the proof of Theorem~\ref{thm: main} in Section~\ref{sec:proofmain}, while we collect all remaining proofs in Section~\ref{sec:proof_prelim} and Appendices~\ref{sec:appa}-\ref{sec:ruiyu_compact}.

\subsection{Examples and discussion of Theorem~\ref{thm: main}}
\label{sec:discussion}

We now highlight several applications of Theorem~\ref{thm: main}. First we remark that for compactly supported distributions, we recover \cite[Theorem 2]{stromme2023minimum}.

\begin{corollary}[Compactly supported distributions] Assume that~$\mu,\nu$ are supported on~$B_1$ and that~$c$ is~$1$-Lipschitz. Then 
\begin{align}
&\E [|\mathcal{C}_{\veps}(\mu_n,\nu_n)-\mathcal{C}_{\veps}(\mu,\nu)|]  \le \frac{C}{\sqrt{n}} (1+\epsilon)\cdot \Bigg(1+e^{\frac{C_{\psi}\tilde{\delta}}{2(t_0/2)^{\gamma}}}\sqrt{\mathcal{N} \Big(\mathrm{spt}(\mu), \frac{\tilde\delta \veps}{2}\Big)\wedge \mathcal{N} \Big(\mathrm{spt}(\nu), \frac{\tilde \delta \veps}{2}\Big)}\Bigg)
\end{align}
for some constant~$C>0$.
\end{corollary}

\begin{proof}
This is a simplified version of Corollary~\ref{cor: first term-before taking expecation} stated below.
\end{proof}

Our next application focuses on subgaussian distributions~$\mu,\nu$. We obtain the following result.

\begin{corollary}[Subgaussian distributions]\label{thm: subGaussian}
Assume that there exist~$\sigma_{\mu},\sigma_{\nu}>0$ such that  Assumption~\ref{assumption:Psialpha} holds with~$\alpha_{\mu}=\alpha_{\nu}=2$ and~$c_{\mu}=\frac{1}{d\sigma_{\mu}^2}$,~$c_{\nu}=\frac{1}{d\sigma_{\nu}^2}$. Define~$\sigma:=\sigma_\mu\vee \sigma_{\nu}$ and let Assumption~\ref{assumption: cost function} hold.
Then
\begin{equation}\label{eq:thmgauss}
\begin{split}
&\E [|\mathcal{C}_{\veps}(\mu_n,\nu_n)-\mathcal{C}_{\veps}(\mu,\nu)|]
\leq \frac{C}{\sqrt{n}} \Big( 1\vee \veps+[(d\sigma^2\vee 1)^2\log(n)]^{\frac{p}{2}}\Big)\\
&\quad  \cdot\Bigg(1+e^{\frac{C_{\psi}\tilde{\delta}}{2(t_0/2)^{\gamma}}}\sqrt{\mathcal{N} \bigg(B_n^\mu, \frac{\tilde\delta  \veps}{C[(d\sigma^2\vee 1)^2\log(n)]^{\frac{p-1}{2}}}\bigg)\wedge \mathcal{N} \bigg(B_n^\nu, \frac{\tilde \delta \veps}{C[(d\sigma^2\vee 1)^2\log(n)]^{\frac{p-1}{2}}}\bigg)}\Bigg)
\end{split}
\end{equation}
holds for all~$n\ge 5$, where the constant~$C>0$ only depends on~$p, C_p, t_0$.
\end{corollary}

\begin{proof}
We note that
\begin{align*}
c_\mu^{-\frac{p}{\alpha_\mu}} &\le  (d\sigma^2)^{\frac{p}{2}},\\
[c_{\mu}^{-1}\paren{c_{\mu}^{-1}\vee 1}\log(n)]^{\frac{p}{\alpha_{\mu}}} &\le [(d\sigma^2\vee 1)^2 \log(n)]^{\frac{p}{2}}, \\ 
r^\mu_n &\le \Big[ 4p(d\sigma^2\vee 1)^2 \paren[\Big]{\frac{p}{2}\vee 1}^2 \log(n)\Big]^{\frac{1}{2}} \le C [(d\sigma^2\vee 1)^2\log(n)]^{\frac{1}{2}}.
\end{align*}
The claim then follows from Theorem~\ref{thm: main}.
\end{proof}

Corollary~\ref{thm: subGaussian} can be further simplified if~$p=2$ and the EOT case~$\varphi(x)=x\log(x)$ or the QOT case~$\varphi(x)=(x^2-1)/2$ holds.

\begin{corollary}[EOT/QOT for subgaussian distributions,~$p=2$]\label{cor:subgaussian}
In the setting of Corollary~\ref{thm: subGaussian}, let~$p=2$ and~$\sigma\ge 1$. Then
\begin{equation*}
\begin{split}
\E[|\mathcal{C}_{\veps}(\mu_n,\nu_n)-\mathcal{C}_{\veps}(\mu,\nu)|]
&\leq \frac{C}{\sqrt{n}} \Big( 1\vee \veps+\frac{Cd^2\sigma^4 \log(n)}{\veps}\Big)^{\frac{d}{2}+1}   
\end{split}
\end{equation*}
holds for all~$n\ge 5,$ where the constant~$C>0$ only depends on~$ C_2$.
\end{corollary}

\begin{proof}
Choosing~$\delta=1/2$,~$t_0= C_{\psi}=1$,~$\gamma=1$ and noting that~$\mathcal{N}(B_r,\epsilon)$ is bounded by~$(1+\frac{2r}{\epsilon})^d$, we obtain
\begin{equation*}
\begin{split}
\mathcal{N} \Big(B_n^\mu, \frac{\veps}{Cd\sigma^2\log(n)^{\frac{1}{2}}}\Big)
&\le \Big(1 + \frac{2r_n^\mu Cd\sigma^2\log(n)^{\frac12}}{\epsilon}\Big)^d\\
&\overset{\mathclap{\eqref{eq:chooseRxRy}}}{\le}  \Big(1+\frac{2C d^2\sigma^4\log(n)}{\veps}\Big)^d,
\end{split}
\end{equation*}   
and similarly for~$\nu.$ The claim follows.
\end{proof}

It is interesting to compare Corollary~\ref{cor:subgaussian} to \cite[Theorem 2]{mena2019statisticalboundsentropicoptimal} on the sample complexity of EOT. Mena and Weed obtain the bound 
\begin{align}\label{eq:mena-weed}
\E[|\mathcal{C}_{\veps}(\mu_n,\nu_n)-\mathcal{C}_{\veps}(\mu,\nu)|]\le \frac{C }{\sqrt{n}} \epsilon\Big(1+\frac{\sigma^{\lceil 5d/2\rceil +6}}{\epsilon^{\lceil 5d/4\rceil +3}}\Big)
\end{align}
for the cost~$c(x,y)=|x-y|^2$ and~$\sigma^2$-subgaussian distributions~$\mu,\nu$, where~$C$ is an unspecified constant depending on~$d.$ Compared to~\eqref{eq:mena-weed}, our rates are less sharp (in~$n$), as they contain an additional factor of~$\log(n)$. However, Corollary~\ref{cor:subgaussian} holds for a much larger class of radial cost functions~$c$ and does not rely on the specific form and smoothness of the quadratic cost. Furthermore, contrary to~\eqref{eq:mena-weed}, we also state the dependence of our rates on the dimension~$d$ explicitly. We also point out that similar results were obtained in \cite[Section 3.5]{groppe2024lower} for subgaussian~$\mu$ and compactly supported~$\nu$ with quadratic cost, building on the approach of \cite{hundrieser2024empirical}. 

Next, similarly to \cite[Example 4,5,6]{stromme2023minimum}, we consider the following setting, that can be formally obtained by setting~$r_n^\nu=r^\nu$ and~$\alpha_\nu=\infty$ in Theorem~\ref{thm: main}:

\begin{corollary}\label{thm:nucompact} Let~$\mu$ satisfy Assumption~\ref{assumption:Psialpha} and assume that there exists~$r^{\nu}>0$, such that~$\supp(\nu)\subseteq B(0,r^{\nu})$.  Furthermore let Assumption~\ref{assumption: cost function} hold. Then
\begin{align}
\E [|\mathcal{C}_{\veps}(\mu_n,\nu_n)-\mathcal{C}_{\veps}(\mu,\nu)|]
&\leq \frac{C}{\sqrt{n}}\Big(1+c_{\mu}^{-\frac{p}{\alpha_{\mu}}}+M_p(\nu)^p\Big)+\frac{C}{\sqrt{n}} \big( \veps+ (r_n^\mu)^{p}\big)\\
&\qquad\cdot\Bigg( 1+e^{\frac{C_{\psi}\tilde{\delta}}{2(t_0/2)^{\gamma}}}\sqrt{\mathcal{N} \bigg(\supp(\nu), \frac{\tilde \delta \veps}{2C_p(r_n^\mu+r^{\nu})^{p-1}}\bigg)}\Bigg).
\end{align}
holds for all~$n\ge 5$, where the constant~$C>0$ only depends on~$\alpha_{\mu}, p, C_p, t_0$.
\end{corollary}

\begin{proof}
see Appendix~\ref{sec:ruiyu_compact}.
\end{proof}

\noindent The following two examples follow directly from Corollary~\ref{thm:nucompact}.

\begin{example}[semi-discrete EOT] \label{eg:semi-discrete}
Assume that~$\mu$ satisfies Assumption~\ref{assumption:Psialpha} and~$\nu$ is supported on~$K$ points. Furthermore let Assumption~\ref{assumption: cost function} hold. Then
\begin{align}
\E [|\mathcal{C}_{\veps}(\mu_n,\nu_n)-\mathcal{C}_{\veps}(\mu,\nu)|]
&\leq \frac{C}{\sqrt{n}}\Big(1+c_{\mu}^{-\frac{p}{\alpha_{\mu}}}+M_p(\nu)^p\Big)+\frac{C}{\sqrt{n}} \big[ \veps+(r_n^\mu)^{p}\big]\Big( 1+e^{\frac{C_{\psi}\tilde{\delta}}{2(t_0/2)^{\gamma}}}\sqrt{K}\Big)
\end{align}
holds for all~$n\ge 5$,  where the constant~$C$ only depends on~$\alpha_{\mu},p, C_p,t_0$ and~$r^\nu$ defined in Corollary~\ref{thm:nucompact}.
\end{example}

\begin{example}[Embedded Manifold]  \label{eg:embedM}
Assume that~$\mu$ satisfies Assumption~\ref{assumption:Psialpha} and~$\nu$ is supported on a~$d_{\nu}$-dimensional, compact, smooth, embedded Riemannian manifold of diameter~$r^\nu$ without boundary. Furthermore, let Assumption~\ref{assumption: cost function} hold. Then 
$\mathcal{N}(\supp(\nu),\delta)\leq C_{\nu}\delta^{-d_{\nu}}$  for some~$C_{\nu}>0$ and~$\delta$ sufficiently small, and consequently, for all~$\epsilon>0$
sufficiently small we have
\begin{align}
\E [|\mathcal{C}_{\veps}(\mu_n,\nu_n)-\mathcal{C}_{\veps}(\mu,\nu)|]
&\leq \frac{C}{\sqrt{n}}\Big(1+c_{\mu}^{-\frac{p}{\alpha_{\mu}}}+M_p(\nu)^p\Big)\\
&\quad +\frac{C}{\sqrt{n}} \big[ \veps+(r_n^\mu)^{p}\big]\Bigg( 1+e^{\frac{C_{\psi}\tilde{\delta}}{2(t_0/2)^{\gamma}}}  \Bigg( \frac{2C_p(r_n^\mu+r^{\nu})^{p-1}}{\tilde \delta\epsilon}\Bigg)^{\frac{d_{\nu}}{2}} \Bigg)
\end{align}
for all~$n\ge 5$, where the constant~$C>0$ only depends on~$ \alpha_\mu, p, C_p, t_0$ and~$C_\nu$.
\end{example}

\begin{proof}
The upper bound on the covering number follows from \cite[Prop. 43, Appendix A]{stromme2023minimum}. Plugging this into Corollary~\ref{thm:nucompact} concludes the proof.
\end{proof}

\section{Preliminary results} \label{sec:preliminary}

In this section we introduce some preliminary results, that will be used in the proof of Theorem~\ref{thm: main}. We defer proofs of these results to Section~\ref{sec:proof_prelim}.

\subsection{Basics}

Recall the definition of~$\mathcal{C}_{\veps}(\mu,\nu, c)$ from~\eqref{eq: ROT}. For future reference let us recall the following fact, that follows directly from the definition:
\begin{equation}\label{eq: scaling with epsilon}
\mathcal{C}_{\veps}(\mu,\nu, c)=\veps \mathcal{C}_1\Big(\mu,\nu,\frac{c}{\veps}\Big).
\end{equation}
We also record the following immediate consequence of Assumption~\ref{assumption: cost function}.
\begin{lemma}\label{lem: pointwise bound of cost function}
Under Assumption~\ref{assumption: cost function} we have
    \begin{equation}\label{eq:estC}
        |c(x,y)|\leq C_p|x-y|^p.
    \end{equation}
\end{lemma}

\subsection{Restriction of probability measures}\label{sec:subsec_compact}
To restrict to probability measures supported on subsets of~$\R^d$, we use the following notation:
\begin{definition}\label{def:restricted}
For a Borel set~$A\subseteq \R^d$ and a probability measure~$\mu\in \mathcal{P}(\R^d)$ we define
\begin{equation}
\mu^{A}( d x):= \frac{1}{\mu(A)}\mathbbm{1}_{A}(x)\mu( d x)
\end{equation}
if~$\mu(A)>0$ and~$\mu^{A} := \delta_0$ otherwise.
For i.i.d. samples~$X_1, \dots, X_n$ drawn from~$\mu$ we define the empirical measure of~$\mu^A$ as 
\begin{equation*}
\mu_n^A:= \frac{1}{|\{i\in\{1, \dots, n\}:  X_i\in A\}|}\sum_{X_i\in A}\delta_{X_i}.
\end{equation*}
if~$|\{i\in\{1, \dots, n\}:  X_i\in A\}|>0$ and~$\mu_n^{A} := \delta_0$ otherwise.
The probability measures~$\nu^A$ and~$\nu_n^A$ are defined similarly.
\end{definition}

\begin{remark}\label{rem:binomial}
As~$X_1, \dots, X_n$ are i.i.d., it is straightforward to see the following: 
\begin{itemize}
    \item~$|\{i\in\{1, \dots, n\}:  X_i\in A\}|\sim \mathrm{Bin}(n, \mu(A)),$ 
    \item conditionally on~$\{|\{i\in\{1, \dots, n\}:  X_i\in A\}|=k\}$,~$\mu^A_n$ is an empirical measure of~$k$ samples of~$\mu^A$. 
\end{itemize}
\end{remark}

\subsection{Regularized optimal transport}
\label{sec:entropic}
In this section we recap basic results on divergence regularized optimal transport. We start with the following well-known duality result.

\begin{lemma}[{Duality, \cite[Proposition 2.3, (i)-(vi)]{gonzalez2025sparse}}]
\label{lem:duality}
Let~$\mu,\nu$ be probability measures on~$\R^d$ with compact supports~$\Omega$ and~$\Omega'$. Let~$c\in C(\Omega\times\Omega')$.  Then
\begin{equation}
\mathcal{C}_{\veps}(\mu,\nu)=\sup\limits_{\hat f\in L^{\infty}(\mu),\hat g\in L^{\infty}(\nu)}\int \hat f\,d\mu+\int \hat g\,d\nu-\veps\int \Big(e^{\frac{\hat f(x)+\hat g(y)-c(x,y)}{\epsilon}}-1\Big)\,\mu(dx)\nu(dy).
\end{equation}
The supremum is attained by the dual potentials~$f\in L^{\infty}(\mu)$,~$g\in L^{\infty}(\nu)$, where we always make the normalization
\begin{align}
\int g\,d\nu= 0.
\end{align}
\end{lemma}
Recalling Definition~\ref{def:restricted} we also define the dual potentials~$f^{r,s}, g^{r,s}$ for~$\mathcal{C}_\epsilon(\mu^{B_r}, \nu^{B_s})$ and~$f^{r,s}_n, g^{r,s}_n$ for~$\mathcal{C}_\epsilon(\mu^{B_r}_n, \nu^{B_s}_n)$. They satisfy the following regularity property.

\begin{lemma}\label{lem: Lip potential}
If Assumption~\ref{assumption: cost function} holds, then~$f^{r,s}$ and~$g^{r,s}$ are~$C_p(r+s)^{p-1}$-Lipschitz.
\end{lemma}

\section{Bounding~$|\mathcal{C}_\epsilon(\mu,\nu)-\mathcal{C}_\epsilon(\mu^{B_r},\nu^{B_s})|$ and its empirical counterpart} \label{sec:compact}

Recalling~\eqref{eq:main_idea_proof}, the aim of this section is to provide bounds on the differences~$$|\mathcal{C}_\epsilon(\mu,\nu)-\mathcal{C}_\epsilon(\mu^{B_r},\nu^{B_s})| \quad \text{and}  \quad \E[|\mathcal{C}_\epsilon(\mu_n,\nu_n)-\mathcal{C}_\epsilon(\mu^{B_r}_n,\nu^{B_s}_n)|]$$ for fixed~$r,s>0$. To achieve this, we first recap general results on the stability of~$\mathcal{C}_\epsilon$ and then specify to our setting. Again we defer proofs to Section~\ref{sec:proof_prelim}.

\subsection{Stability of regularized optimal transport}
We make use of the following results from \cite{eckstein2022quantitativestabilityregularizedoptimal} on stability of regularized optimal transport. 

\begin{definition}[cf.~{\cite[Definition 3.3]{eckstein2022quantitativestabilityregularizedoptimal}}] \label{def:Al}
Let~$p\ge 1, L>0$ and let~$\mu_i,\tilde{\mu}_i\in\mathcal{P}_p(\R^d)$ for~$i=1,2$. We say a function~$ c$ satisfies~\eqref{eq: AL} if
\begin{equation}\label{eq: AL} \tag{$A_L$}
\bigg|\int  c\, d(\pi-\tilde{\pi})\bigg|\leq LW_p(\pi,\tilde{\pi})
\end{equation}
for all~$\pi\in \Pi(\mu_1,\mu_2), \tilde{\pi}\in\Pi(\tilde{\mu}_1,\tilde{\mu}_2)$. Here~$W_p$ is the Wasserstein distance wrt.~the norm~$(|\cdot|^p+|\cdot|^p)^{1/p}$ on~$\R^d\times \R^d.$
\end{definition}

\begin{theorem}[cf.~{\cite[Theorem 3.7]{eckstein2022quantitativestabilityregularizedoptimal}}] \label{thm: continue value of EOT}Let~$p\ge 1$. Let ~$\mu_i,\tilde{\mu}_i\in\mathcal{P}_p(\R^d), i=1,2$ and let~$c$ satisfy~\eqref{eq: AL}. Then
\begin{equation}\label{eq:EOTstablility}
|\mathcal{C}_1(\mu_1,\mu_2)-\mathcal{C}_1(\tilde{\mu}_1,\tilde{\mu}_2) |\leq L [W_p(\mu_1,\tilde{\mu}_1)^p +W_p(\mu_2,\tilde{\mu}_2)^p]^{1/p} =:LW_p(\mu_1, \mu_2;\tilde \mu_1, \tilde\mu_2).
\end{equation}
\end{theorem}

The following lemma is a variation of \cite[Proof of Example 3.6]{eckstein2022quantitativestabilityregularizedoptimal}. 

\begin{lemma}\label{lem: AL}
For a cost function~$c$ satisfying Assumption~\ref{assumption: cost function},~\eqref{eq: AL} holds with
\begin{equation}\label{eq:defL}
L=C\Big [M_p(\mu_1)+M_p(\mu_2)+M_p(\tilde{\mu}_1)+M_p(\tilde{\mu}_2)\Big]^{p-1},
\end{equation}
where we recall~$M_p(\nu) =(\int\|x\|^p\, \nu(dx))^{1/p}$ for~$\nu\in\mathcal{P}(\mathbb{R}^d)$, and~$C$ is a constant only depending  on~$p$ and~$C_p$.
\end{lemma}

\subsection{Bounding ~$|\mathcal{C}_\epsilon(\mu,\nu)-\mathcal{C}_\epsilon(\mu^{B_r},\nu^{B_s})|$}

For the remainder of this section we assume that Assumptions~\ref{assumption:Psialpha} and~\ref{assumption: cost function} are in force. We also fix~$r,s>0$ and recall~$\mu^{B_r},\nu^{B_s}$ from Definition~\ref{def:restricted}.

\begin{lemma}[Scaled cost]\label{lem: third term L}
We have
\begin{equation}
\Big|\int \frac{c}{\veps}\, d(\pi-\tilde{\pi})\Big|\leq LW_p(\pi,\tilde{\pi})
\end{equation}
for all~$\pi\in \Pi(\mu,\nu)$ and~$\tilde{\pi}\in\Pi(\mu^{B_r},\nu^{B_s})$, where
\begin{equation}\label{eq:LmuBnuB}
L=\frac{C}{\veps}\paren[\big]{ M_p(\mu)+M_p(\nu)}^{p-1}.%
\end{equation}
Here the constant~$C$ only depends on~$p$ and~$C_p$.
\end{lemma}

\begin{lemma}\label{lem: w2 of true approx}
We have
\begin{align}
W_p(\mu,\nu; \mu^{B_r}, \nu^{B_s})^p
&\le 2^{p-1}\Big[\mu(B_r^c)\paren[\big]{M_p(\mu^{B_r})^p+M_p(\mu^{B_r^c})^p}\\
&\qquad\quad + \nu(B_s^c)\paren[\big]{M_p(\nu^{B_s})^p+M_p(\nu^{B_s^c})^p}\Big]. \label{eq:WpApprox}
\end{align} 
\end{lemma}

Combining Lemma~\ref{lem: third term L} and Lemma~\ref{lem: w2 of true approx} with Theorem~\ref{thm: continue value of EOT} immediately gives the following lemma.
\begin{lemma}\label{lem: third error in epsilon case}
We have 
    \begin{align}
    \abs[\big]{\mathcal{C}_{\veps}(\mu,\nu)-\mathcal{C}_{\veps}(\mu^{B_r},\nu^{B_s})}&\leq C\paren[\big]{ M_p(\mu)+M_p(\nu)}^{p-1}\Big[\mu(B_r^c)\paren[\big]{M_p(\mu^{B_r})^p+M_p(\mu^{B_r^c})^p}\\
&\qquad + \nu(B_s^c)\paren[\big]{M_p(\nu^{B_s})^p+M_p(\nu^{B_s^c})^p}\Big] ^{\frac1p}, \label{eq: estimate third term}
\end{align}
where the constant~$C$ only depends  on~$p$ and~$C_p$.
\end{lemma}

\subsection{Bounding  ~$\E[|\mathcal{C}_\epsilon(\mu_n,\nu_n)-\mathcal{C}_\epsilon(\mu^{B_r}_n,\nu^{B_s}_n)|]$}
We now carry out a similar analysis for~$\mu_n, \nu_n.$ For notational simplicity we set~$$n_r:=|\{i\in\{1, \dots, n\}:  X_i\in B_r\}|, \quad n_s :=|\{i\in\{1, \dots, n\}:  Y_i\in B_s\}|.$$

\begin{lemma}\label{lem: estimate second term-before take the expetation}
If~$n_r, n_s>0$, then we have
\begin{align*}
\MoveEqLeft\abs{\mathcal{C}_{\veps}(\mu_n,\nu_n)-\mathcal{C}_{\veps}(\mu_n^{B_r},\nu_n^{B_s})}\leq  C\big[M_p(\mu_n)+M_p(\nu_n)+M_p(\mu_n^{B_r})+M_p(\nu_n^{B_s})\big]^{p-1}\\
&\cdot \Big[ \Big(\frac{1}{n_r}-\frac{1}{n}\Big) \sum_{X_i\in B_r}|X_i|^p+\frac{1}{n}\sum_{X_i\notin B_r}|X_i|^p+\Big(\frac{1}{n_s}-\frac{1}{n}\Big) \sum_{Y_i\in B_s}|Y_i|^p+\frac{1}{n}\sum_{Y_i\notin B_s}|Y_i|^p\Big]^{\frac1p},
\end{align*}
where the constant~$C$ only depends on~$p$ and~$C_p$.
\end{lemma}

Taking the conditional expectation on both sides of Lemma~\ref{lem: estimate second term-before take the expetation}, we have the following result.

\begin{lemma}\label{lem: estimate of second term-after expectation} If~$n_r, n_s>0$, then we have
\begin{align}
\MoveEqLeft\E\big[|\mathcal{C}_{\veps}(\mu_n,\nu_n)-\mathcal{C}_{\veps}(\mu_n^{B_r},\nu_n^{B_s})|\,\big| n_r,n_s\big]
\leq C\Big(M_p(\mu)^p+M_p(\nu)^p\Big)^{\frac{p-1}{p}}\\
& \cdot\bigg[\paren[\Big]{M_p(\mu)^p+M_p(\mu^{B_r^c})^p}\cdot  \Big(1-\frac{n_r}{n}\Big) 
+\paren[\Big]{M_p(\nu)^p+M_p(\nu^{B_s^c})^p}\cdot\Big(1-\frac{n_s}{n}\Big)\bigg]^{\frac1p} \label{eq:ApproxOfEmpCpt}
\end{align}
where the constant~$C$ depends on~$p$ and~$C_p$.
\end{lemma}

\section{Bounding~$\E[|\mathcal{C}_{\veps}(\mu^{B_r}_n,\nu^{B_s}_n)-\mathcal{C}_{\veps}(\mu^{B_r},\nu^{B_s})|]$} \label{sec:stromme}

We now bound the middle term in~\eqref{eq:main_idea_proof}. For this we use the following result, which is a generalization of the result in \cite{stromme2023minimum} to divergence regularized optimal transport.
\begin{lemma} \label{lem: estimate first term-before taking expectation}
Define the
population density
\begin{equation}\label{eq:pb}
 p_{\veps}(x,y)=\psi'\Big(\frac{f_{\veps}(x)+g_{\veps}(y)-c(x,y)}{\veps}\Big).
\end{equation}
If~$n_r,n_s >0$, then we have
\begin{align}
&\E \big[|\mathcal{C}_{\veps}(\mu_n^{B_r},\nu_n^{B_s})-\mathcal{C}_{\veps}(\mu^{B_r},\nu^{B_s})|\,\big| n_r, n_s\big]\\
& \hspace{2cm} \leq \sqrt{\frac{\var_{\mu^{B_r}}(f^{r,s})}{n_r}}+\sqrt{\frac{\var_{\nu^{B_s}}(g^{r,s})}{n_s}}+\frac{\veps}{\sqrt{n_rn_s}}\norm[\big]{p^{r,s}}_{L^2(\mu^{B_r}\otimes\nu^{B_s})}\\
&\hspace{4cm} + \frac{\sqrt{2}\|p^{r,s}\|_{L^2(\mu^{B_r}\otimes\nu^{B_s})}}{(n_rn_s)^{\frac14}}\E\Big[\|(f^{r,s}_n-f^{r,s},g_n^{r,s}-g^{r,s})\|^2_{L^2(\mu_n^{B_r})\times L^2(\nu_n^{B_s})}\big| n_r, n_s\Big]^{\frac12}.
\label{eq:ROTCpt}
\end{align}
\end{lemma}

\subsection{Norm of entropic densities~$p^{r,s}$}
It remains to bound the density~$p^{r,s}$ in the space~$L^2(\mu^{B_r}\otimes \nu^{B_s})$. For this we define~$$B^\mu_r:= B_r\cap \mathrm{spt}(\mu), \quad \quad B^\nu_s:= B_s\cap \mathrm{spt}(\nu),$$ and use the following result.

\begin{lemma}[Estimation of density via covering numbers]\label{lem:L2normP} We have
\begin{align}
\norm{p^{r,s}(x,y)}_{L^2(\mu^{B_r}\otimes \nu^{B_s})}\leq 1+e^{\frac{C_{\psi}\tilde{\delta}}{2(t_0/2)^{\gamma}}}\sqrt{\mathcal{N} \Big(B^\mu_r, \frac{\tilde{\delta}\veps}{2C_p(r+s)^{p-1}}\Big)\wedge \mathcal{N}\Big(B^\nu_s, \frac{\tilde{\delta}\veps}{2C_p(r+s)^{p-1}}\Big)},
\end{align}
where we recall~$\tilde{\delta}=\min\{\delta, t_0/2\}$ from~\eqref{eq:delta}.
\end{lemma}

Applying Lemma~\ref{lem:LinftyBoundSchrondinger}, Lemma~\ref{lem:varfg} and Lemma~\ref{lem:L2normP} to Lemma~\ref{lem: estimate first term-before taking expectation} yields following corollary.

\begin{corollary}\label{cor: first term-before taking expecation} If Assumption~\ref{assumption: cost function} holds and~$n_r, n_s>0$, then we have
\begin{align}
\begin{split}
\E \big[|\mathcal{C}_{\veps}(\mu_n^{B_r},\nu_n^{B_s})&-\mathcal{C}_{\veps}(\mu^{B_r},\nu^{B_s})|\,\big| n_r, n_s\big]\\
&\leq \frac{C_p(r+s)^p}{\sqrt{n_r}}+\frac{C_p(r+s)^p}{\sqrt{n_s}}
 + \bigg[\frac{ 4(t_0+9C_p(r+s)^p)}{(n_rn_s)^{\frac14}}+\frac{\veps}{\sqrt{n_rn_s}}\bigg]\\
&\quad \cdot\paren[\bigg]{1+e^{\frac{C_{\psi}\tilde{\delta}}{2(t_0/2)^{\gamma}}}\sqrt{\mathcal{N} \Big(B^\mu_r, \frac{\tilde{\delta}\veps}{2C_p(r+s)^{p-1}}\Big)\wedge \mathcal{N}\Big(B^\nu_s, \frac{\tilde{\delta}\veps}{2C_p(r+s)^{p-1}}\Big)}}. \label{eq:EmpApproxCpt}
\end{split}
\end{align}
\end{corollary}

\section{Proof of Theorem~\ref{thm: main}} \label{sec:proofmain}

Throughout this section, we assume that Assumptions~\ref{assumption:Psialpha}-\ref{assum:psi} are in force.  We first state two additional estimates for ease of reference in the proof of Theorem~\ref{thm: main}.

\begin{lemma}\label{lem: nx>0, or ny>0} For~$i,j\geq 0$ and~$r,s>0$ we have
\begin{align}\label{eq:nx>0}
\E \big[|\mathcal{C}_{\veps}(\mu_n,\nu_n)-\mathcal{C}_{\veps}(\mu,\nu)|\,\big|\,n_r=i,n_s=0\big]
\leq   C\bigg[ 1+\Big(1+\frac{i}{n}\Big)M_p(\mu)^p &+ \paren[\Big]{1-\frac{i}{n}}M_p(\mu^{B_r^c})^p\nonumber \\
&+M_p(\nu)^p+M_p(\nu^{B_s^c})^p\bigg],
\end{align}
as well as
\begin{align}\label{eq:ny>0}
\E \big[ |\mathcal{C}_{\veps}(\mu_n,\nu_n)-\mathcal{C}_{\veps}(\mu,\nu)|\,\big|\,n_r=0,n_s=j\big] 
\leq  C\bigg[ 1+\Big( 1+\frac{j}{n}\Big)M_p(\nu)^p &+\Big( 1-\frac{j}{n}\Big)M_p(\nu^{B_s^c})^p\nonumber \\
& +M_p(\mu)^p+M_p(\mu^{B_r^c})^p\bigg].
\end{align}
Here the constant~$C$ only depends on~$p$ and~$C_p$.
\end{lemma}

\begin{lemma}\label{lem: expectation of empirical approx}
For any~$r,s>0$ we have
\begin{align}
&\sum_{i,j=1}^{n}\E \big[ |\mathcal{C}_{\veps}(\mu_n^{B_r},\nu_n^{B_s})-\mathcal{C}_{\veps}(\mu_n,\nu_n)|\,\big| n_r=i,n_s=j \big]\cdot \mathbb{P}(n_r=i,n_s=j)\\
&\hspace{2cm}\leq C\Big( M_p(\mu)^p+M_p(\nu)^p\Big)^{\frac{p-1}{p}}\cdot\bigg[ \paren[\Big]{M_p(\mu)^p+M_p(\mu^{B_r^c})^p}\mu(B_r^c)\\
&\hspace{8cm}+\paren[\Big]{M_p(\nu)^p+M_p(\nu^{B_s^c})^p}\nu(B_s^c)\bigg]^{\frac1p}.
\label{eq:Enr>0ns>0}
\end{align}
Here the constant~$C$ only depends on~$p$ and~$C_p$.
\end{lemma}

The following lemma explains the choices~$r=r_n^\mu$ and~$s=r_n^\nu$ in the proof of Theorem~\ref{thm: main} below.

\begin{lemma}[Choice of Truncated Sets] \label{lem: choiceRxRy}
If~$r_n^\mu, r_n^\nu$ are chosen as in~\eqref{eq:chooseRxRy} and~$n\geq 5$,
then  
\begin{align}\label{eq:GammaChooseRS}
\begin{split}
\mu((B_n^\mu)^c)\cdot M_p(\mu^{(B_n^\mu)^c})^p&\le \frac{2}{n^{\frac{p}{2}}}\paren[\Big]{1+\frac{p}{\alpha_{\mu}}c_{\mu}^{-\frac{p}{\alpha_{\mu}}}},\\
\nu(\mathbb (B_n^\nu)^c)\cdot M_p(\nu^{(B_n^\nu)^c})^p&\le \frac{2}{n^{\frac{p}{2}}}\paren[\Big]{1+\frac{p}{\alpha_{\nu}}c_{\nu}^{-\frac{p}{\alpha_{\nu}}}},
\end{split}
\end{align}
and
\begin{equation}
\mu((B_n^\mu)^c)\le  \frac{2}{n^p},\quad \nu(\mathbb (B_n^\nu)^c)\le \frac{2}{n^p}. \label{eq:TailChooseRS}
\end{equation}
Furthermore,
\begin{align}\label{eq:IntGamma1}
M_p( \mu)^p\leq \frac{2p}{\alpha_{\mu}}c_{\mu}^{-\frac{p}{\alpha_{\mu}}}\Gamma\Big(\frac{p}{\alpha_{\mu}}\Big), \quad M_p( \nu)^p\leq \frac{2p}{\alpha_{\nu}}c_{\nu}^{-\frac{p}{\alpha_{\nu}}}\Gamma\Big(\frac{p}{\alpha_{\nu}}\Big).
\end{align}
\end{lemma}

We are now in a position for the proof of our main result, Theorem~\ref{thm: main}. Throughout we make the convention, that the constant~$C$ only depends on~$p,\alpha_{\mu}, \alpha_{\nu}, C_p~$ and may change from line to line.

\begin{proof}[Proof of Theorem~\ref{thm: main}]
Fix~$n\geq 5$ and choose~$r=r_n^{\mu}, s=r_n^{\nu}$,
where~$r_n^{\mu}$ and~$r_n^{\nu}$ are defined as in~\eqref{eq:chooseRxRy} --- to improve readability, we continue to write~$r,s$ throughout the proof.  Recalling~\eqref{eq:main_idea_proof}
we use the tower property of conditional expectation to obtain
\begin{equation}\label{eq: condition on samples inside compact set}
\E [|\mathcal{C}_{\veps}(\mu_n,\nu_n)-\mathcal{C}_{\veps}(\mu,\nu)|]
=\E\big[\E \big[|\mathcal{C}_{\veps}(\mu_n,\nu_n)-\mathcal{C}_{\veps}(\mu,\nu)|\,\big|n_r,n_s\big]\big]=T_1+T_2+T_3+T_4,
\end{equation}
where
\begin{align}
T_1&:=\E \big[\mathcal{C}_{\veps}(\mu_n,\nu_n)-\mathcal{C}_{\veps}(\mu,\nu)|\,\big| n_r=0,n_s=0\big]\cdot \mathbb{P}(n_r=0,n_s=0),\\
T_2&:=\sum_{j=1}^{n}\E \big[|\mathcal{C}_{\veps}(\mu_n,\nu_n)-\mathcal{C}_{\veps}(\mu,\nu)|\,\big|n_r=0,n_s=j\big]\cdot \mathbb{P}(n_r=0,n_s=j),\\
T_3&:=\sum_{i=1}^{n}\E \big[|\mathcal{C}_{\veps}(\mu_n,\nu_n)-\mathcal{C}_{\veps}(\mu,\nu)|\,\big| n_r=i,n_s=0\big]\cdot \mathbb{P}(n_r=i,n_s=0),\\
T_4&:=\sum_{i,j=1}^{n}\E \big[|\mathcal{C}_{\veps}(\mu_n,\nu_n)-\mathcal{C}_{\veps}(\mu,\nu)|\,\big| n_r=i,n_s=j\big]\cdot \mathbb{P}(n_r=i,n_s=j).
\end{align}
We bound the four terms~$T_1, T_2, T_3, T_4$ separately. For this we first recall from Remark~\ref{rem:binomial}, that~$n_r\sim \text{Bin}(n, \mu(B_n^\mu))$ and~$n_s\sim \text{Bin}(n, \nu(B_n^\nu))$ are independent, and thus
\begin{align}\label{eq:binom2}
\mathbb P(n_r=i, n_s=j) = C_n^i \mu(B_n^\mu)^i \mu((B_n^\mu)^c)^{n-i}\cdot  C_n^j \nu(B_n^\nu)^j \nu(B_n^\nu)^{n-j},
\end{align}
where~$C_n^i:=\binom{i}{n}$.
\restartsteps

\step[Bounding~$T_1+T_2+T_3$]
For term~$T_1$, we use~\eqref{eq:binom2} to see that
\begin{equation*}
    \mathbb{P}(n_r=0,n_s=0)=\mu((B_n^\mu)^c)^n\cdot \nu(\mathbb (B_n^\nu)^c)^n,
\end{equation*}
and obtain
\begin{align}
T_1&\overset{\mathclap{\eqref{eq:nx>0}}}{\leq}  C\Big(1+ M_p(\mu)^p+M_p(\nu)^p\Big)\mu((B_n^\mu)^c)^n\cdot \nu(\mathbb (B_n^\nu)^c)^n\\
&\qquad+ C\paren[\Big]{\mu((B_n^\mu)^c)\cdot M_p(\mu^{(B_n^\mu)^c})^p}\mu((B_n^\mu)^c)^{n-1}\cdot \nu(\mathbb (B_n^\nu)^c)^n\\
&\qquad+ C\paren[\Big]{\nu(\mathbb (B_n^\nu)^c)\cdot M_p(\nu^{(B_n^\nu)^c})^p}\mu((B_n^\mu)^c)^n\cdot \nu(\mathbb (B_n^\nu)^c)^{n-1}.
\label{ineq:T1}
\end{align}
We now turn to~$T_2, T_3$. By Lemma~\ref{lem: nx>0, or ny>0} and~\eqref{eq:binom2} we have
\begin{align}
T_2&\, \overset{\mathclap{\eqref{eq:ny>0}}}{\leq} C\sum_{j=1}^{n} \mathbb{P}(n_r=0,n_s=j)\cdot \bigg[  1+\Big( 1+\frac{j}{n}\Big)M_p(\nu)^p +\Big( 1-\frac{j}{n}\Big)M_p(\nu^{B_s^c})^p+M_p(\mu)^p+M_p(\mu^{B_r^c})^p\bigg]\\
&\leq C\bigg[ 1+2M_p(\nu)^p+\nu(\mathbb (B_n^\nu)^c)\cdot M_p(\nu^{(B_n^\nu)^c})^p
+M_p(\mu)^p+M_p(\mu^{(B_n^\mu)^c})^p\bigg]\mu((B_n^\mu)^c)^n, \label{ineq:T2}
\end{align}
where in the last inequality we use the fact that
\begin{equation}
    \sum_{j=1}^{n}\mathbb{P}(n_s=j)\Big( 1-\frac{j}{n}\Big) \leq \E \Big( 1-\frac{n_s}{n}\Big) =\nu(\mathbb (B_n^\nu)^c).
\end{equation}
By symmetry,
\begin{equation}\label{ineq:T3}
T_3\overset{\mathclap{\eqref{eq:nx>0}}}{\leq} C\bigg[ 1+2M_p(\mu)^p+\mu((B_n^\mu)^c)\cdot M_p(\mu^{(B_n^\mu)^c})^p
+M_p(\nu)^p+M_p(\nu^{(B_n^\nu)^c})^p\bigg]\nu(\mathbb (B_n^\nu)^c)^n. 
\end{equation}
Summing up~$T_1, T_2$ and~$T_3$ using~\eqref{ineq:T1},\eqref{ineq:T2},\eqref{ineq:T3}, we obtain by direct computation 
\begin{equation}
T_1+T_2+T_3
\overset{\eqref{eq:GammaChooseRS}-\eqref{eq:IntGamma1}}{\leq}\quad C \Big( 1+c_{\mu}^{-\frac{p}{\alpha_{\mu}}}+c_{\nu}^{-\frac{p}{\alpha_{\nu}}}\Big)\frac{1}{n^{\frac{p}{2}}}\Big(\frac{2}{n^p}\Big)^{n-1},\label{eq:sumT1T2T3}
\end{equation}
where the constant~$C$ only depends on~$p,\alpha_{\mu}$ and~$\alpha_{\nu}, C_p~$.
\step[Bounding~$T_4$] By the triangle inequality,
\begin{align}\label{eq: step3}
\begin{split}
\E \big[ |\mathcal{C}_{\veps}(\mu_n,\nu_n)-\mathcal{C}_{\veps}(\mu,\nu)|\,\big| n_r,n_s\big]\le &\E\big[|\mathcal{C}_{\veps}(\mu_n,\nu_n)-\mathcal{C}_{\veps}(\mu_n^{B_r},\nu_n^{B_s})|\,\big| n_r,n_s]\\
&+\E \big[|\mathcal{C}_{\veps}(\mu_n^{B_r},\nu_n^{B_s})-\mathcal{C}_{\veps}(\mu^{B_r},\nu^{B_s})|\,\big| n_r,n_s\big]\\
&+|\mathcal{C}_{\veps}(\mu^{B_r},\nu^{B_s})-\mathcal{C}_{\veps}(\mu,\nu)|.
\end{split}
\end{align}
We now bound the three terms on the right-hand side of~\eqref{eq: step3} separately.
For the first term of~\eqref{eq: step3} we use Lemma~\ref{lem: expectation of empirical approx} and Lemma~\ref{l:estGamma} to estimate
 \begin{align}
 \begin{split}
&\sum_{i,j=1}^{n}\E \big[|\mathcal{C}_{\veps}(\mu_n,\nu_n)-\mathcal{C}_{\veps}(\mu_n^{B_r},\nu_n^{B_s})|\,\big| n_r=i,n_s=j\big]\cdot \mathbb{P}(n_r=i,n_s=j)\\
&\qquad\overset{\mathclap{\eqref{eq:Enr>0ns>0}}}{\leq} C\Big( M_p(\mu)^p+M_p(\nu)^p\Big)^{\frac{p-1}{p}}\cdot\bigg[ \paren[\Big]{M_p(\mu)^p+M_p(\mu^{B_r^c})^p}\mu(B_r^c)\\
&\hspace{6cm}+\paren[\Big]{M_p(\nu)^p+M_p(\nu^{B_s^c})^p}\nu(B_s^c)\bigg]^{\frac1p}\\
&\qquad\overset{\mathclap{\eqref{eq:GammaChooseRS}-\eqref{eq:IntGamma1}}}{\leq} \quad 
\frac{C}{\sqrt{n}}\Big(1+c_{\mu}^{-\frac{p}{\alpha_{\mu}}} +c_{\nu}^{-\frac{p}{\alpha_{\nu}}}\Big).\label{eq:T43}
\end{split}
\end{align}
We now estimate the second term on the right hand side of~\eqref{eq: step3}. For this we first note that by~\eqref{eq:binom2} and Lemma~\ref{lem: binomial integrals} with~$a=\mu(B_n^\mu)\ge 1-2/n^2$ resp.~$a=\nu(B_n^\nu)\ge 1-2/n^2$ recalling~\eqref{eq:TailChooseRS} we have
\begin{align}\label{eq:binom_est}
\sum_{i=1}^n \frac{1}{\sqrt{i}} \cdot \mathbb P(n_r=i)
=\sum_{i=1}^n \frac{1}{\sqrt{i}} C^i_n a^i(1-a)^{n-i}
&\le \frac{C}{\sqrt{n}},\qquad
\sum_{i=1}^n \frac{1}{\sqrt[4]{i}} \cdot \mathbb P(n_r=i) \le \frac{C}{\sqrt[4]{n}},
\end{align}
and similarly for~$n_s.$
By Corollary~\ref{cor: first term-before taking expecation} we then conclude for~$n\geq 5$
\begin{align}
&\sum_{i,j=1}^{n}\E \big[|\mathcal{C}_{\veps}(\mu_n^{B_r},\nu_n^{B_s})-\mathcal{C}_{\veps}(\mu^{B_r},\nu^{B_s})|\,\big| n_r=i,n_s=j\big]\cdot \mathbb{P}(n_r=i,n_s=j)\\
&\overset{\mathclap{\eqref{eq:EmpApproxCpt}}}{\leq}\Big[C_p(r+s)^p\Big] \sum_{i,j=1}^{n}\Big(\frac{1}{\sqrt{i}}+\frac{1}{\sqrt{j}}\Big)\cdot \mathbb{P}(n_r=i,n_s=j)
\\
&\quad+\paren[\bigg]{1+e^{\frac{C_{\psi}\tilde{\delta}}{2(t_0/2)^{\gamma}}}\sqrt{\mathcal{N} \Big(B^\mu_r, \frac{\tilde{\delta}\veps}{2C_p(r+s)^{p-1}}\Big)\wedge \mathcal{N}\Big(B^\nu_s, \frac{\tilde{\delta}\veps}{2C_p(r+s)^{p-1}}\Big)}}\\
&\qquad\cdot\sum_{i,j=1}^{n}\bigg[\frac{4\big(t_0+9C_p(r+s)^p\big)}{(ij)^{\frac14}}+\frac{\veps}{\sqrt{ij}}\bigg]\cdot \mathbb{P}(n_r=i,n_s=j)\\
&\overset{\mathclap{\eqref{eq:binom_est}}}{\leq} \frac{C}{\sqrt{n}} \cdot \bigg[C_p(r+s)^p\bigg]+C\bigg[\frac{C_p(r+s)^p}{\sqrt{n}}+\frac{\veps}{n}\bigg]\\
&\quad \cdot\paren[\bigg]{1+e^{\frac{C_{\psi}\tilde{\delta}}{2(t_0/2)^{\gamma}}}\sqrt{\mathcal{N} \Big(B^\mu_r, \frac{\tilde{\delta}\veps}{2C_p(r+s)^{p-1}}\Big)\wedge \mathcal{N}\Big(B^\nu_s, \frac{\tilde{\delta}\veps}{2C_p(r+s)^{p-1}}\Big)}}. \label{eq:T42}
\end{align}
For the last term on the right hand side of~\eqref{eq: step3}, 
\begin{align}
\begin{split}
|\mathcal{C}_{\veps}(\mu^{B_r},\nu^{B_s})-\mathcal{C}_{\veps}(\mu,\nu)|&\overset{\mathclap{\eqref{eq: estimate third term}}}{\leq}  C\paren[\big]{ M_p(\mu)+M_p(\nu)}^{p-1}\Big[\mu(B_r^c)\paren[\big]{M_p(\mu^{B_r})^p+M_p(\mu^{B_r^c})^p}\\
&\hspace{4.5cm} + \nu(B_s^c)\paren[\big]{M_p(\nu^{B_s})^p+M_p(\nu^{B_s^c})^p}\Big] ^{\frac1p}\\
&\overset{\mathclap{\eqref{eq:GammaChooseRS}-\eqref{eq:IntGamma1}}}{\leq}\qquad  \frac{C}{\sqrt{n}}\Big(1+c_{\mu}^{-\frac{p}{\alpha_{\mu}}}+c_{\nu}^{-\frac{p}{\alpha_{\nu}}}\Big). 
\label{eq:T41}
\end{split}
\end{align}
Thus, we obtain 
\begin{align}
\begin{split}
T_4\qquad\overset{\mathclap{\eqref{eq:T43}-\eqref{eq:T41}}}{\leq} \qquad &\frac{C}{\sqrt{n}} \Big(1+c_{\mu}^{-\frac{p}{\alpha_{\mu}}}+c_{\nu}^{-\frac{p}{\alpha_{\nu}}}\Big)
+\frac{C}{\sqrt{n}}+\Big(\frac{C(r+s)^p}{\sqrt{n}}+\frac{C\veps}{n}\Big)\\
&\quad \cdot 
        \paren[\bigg]{1+e^{\frac{C_{\psi}\tilde{\delta}}{2(t_0/2)^{\gamma}}}\sqrt{\mathcal{N} \Big(B^\mu_r, \frac{\tilde{\delta}\veps}{2C_p(r+s)^{p-1}}\Big)\wedge \mathcal{N}\Big(B^\nu_s, \frac{\tilde{\delta}\veps}{2C_p(r+s)^{p-1}}\Big)}}. \label{eq:T4}
\end{split}
\end{align}
Combining~\eqref{eq:chooseRxRy} with~\eqref{eq:sumT1T2T3} and~\eqref{eq:T4}   completes the proof.
\end{proof}

\section{Proof of auxiliary results}
\label{sec:proof_prelim}

\subsection{Remaining proofs from Section~\ref{sec:preliminary}}
Lemma~\ref{lem: pointwise bound of cost function} follows immediately from Assumption~\ref{assumption: cost function}. 
\begin{proof}[Proof of Lemma~\ref{lem: pointwise bound of cost function}]
    Since~$h$ satisfies~\eqref{eq:devh}, we conclude for~$t\ge0$
\begin{equation*}
    |h(t)-h(0)|\le C_p t^{p-1} |t| =C_p t^p,
\end{equation*}
as claimed.
\end{proof}

\begin{proof}[Proof of Lemma~\ref{lem: Lip potential}]
According to \cite[Proposition 2.3, (\romannumeral 6)]{gonzalez2025sparse},~$f^{r,s}$ and~$g^{r,s}$ share the same moduli of continuity with~$c(x,y)$ on~$B_r\times B_s$. It is straightforward to see that for any~$x,x'\in B_s$ and any~$y\in B_s$ we have 
\begin{equation}
\abs{c(x,y)-c(x',y)} =\abs[\big]{h(|x-y|)-h(|x'-y|)}\overset{\eqref{eq:devh}}{\leq}  C_p(r+s)^{p-1}|x-x'|. \label{eq:Lipc}
\end{equation}
Therefore,
\begin{equation}
    \abs{f^{r,s}(x)-f^{r,s}(x')}\overset{\eqref{eq:Lipc}}{\leq} C_p(r+s)^{p-1}|x-x'|.
\end{equation}
and analogously, for any~$y,y'\in B_s$,
\begin{equation*}
 \abs{g^{r,s}(y)-g^{r,s}(y')}\leq C_p(r+s)^{p-1}|y-y'|.\qedhere
\end{equation*}
\end{proof}

\subsection{Remaining proofs from Section~\ref{sec:compact}}

\begin{proof}[Proof of Lemma~\ref{lem: AL}]
We first set up some notation: recalling that~$W_p$ is the~$p$-Wasserstein distance wrt.~the norm~$(|\cdot|^p+|\cdot|^p)^{1/p},$ let
\[
\kappa=\kappa(dx_1,d x_2,dy_1,d y_2)
\]
be a~$W_p$-optimal coupling between~$\pi( dx_1,dx_2)$ and~$\tilde{\pi}(dy_1,dy_2)$, where~$x_1, x_2, y_1, y_2\in \R^d$. To shorten notation we write~$x:=(x_1, x_2)\in \R^d\times \R^d$ and~$y:=(y_1,y_2)\in \R^d\times \R^d$. Now we observe that
\begin{align}
\begin{split}
\abs[\bigg]{\int c\, d\pi-\int c\, d\tilde{\pi}} &=\abs[\bigg]{\int h(\abs{x_1-x_2})\, \kappa(dx,d y)-\int h(\abs{y_1-y_2})\, \kappa(dx,d y)}\\
    &\overset{\mathclap{\eqref{eq:devh}}}{\leq} \int C_p \Big(\abs{x_2-x_1}\vee \abs{y_2-y_1}\Big)^{p-1}\abs[\Big]{\abs{x_2-x_1}-\abs{y_2-y_1}}\,\kappa(dx,d y)\\
    &\overset{\mathclap{\text{H\"older's}}}{\leq}\quad C_p \bigg( \int \Big(\abs{x_2-x_1}\vee \abs{y_2-y_1}\Big)^{p}\kappa(dx,d y)\bigg)^{\frac{p-1}{p}}\\
    &\hspace{2cm}\cdot\bigg( \int \abs[\Big]{\abs{x_2-x_1}-\abs{y_2-y_1}}^p\kappa(dx,d y)\bigg)^{\frac{1}{p}}. \label{eq:holderLWp}
\end{split}
\end{align}
Next we bound the two terms on the right hand side of~\eqref{eq:holderLWp}.
For the first term we use Minkowski's inequality to estimate
\begin{equation}
\bigg( \int \Big(\abs{x_2-x_1}\vee \abs{y_2-y_1}\Big)^{p}\kappa(dx,d y)\bigg)^{\frac{p-1}{p}} \leq\Big[M_p(\mu_1)+M_p(\mu_2)+M_p(\tilde{\mu}_1)+M_p(\tilde{\mu}_2)\Big]^{p-1}.\label{eq:term1L}
\end{equation}
For the second term, using the fact that
\begin{equation*}
    \abs[\Big]{\abs{x_2-x_1}-\abs{y_2-y_1}}\leq \abs[\big]{(x_2-x_1)-(y_2-y_1)}\leq \abs{x_1-y_1}+\abs{x_2-y_2},
\end{equation*}
we obtain
\begin{align}
\begin{split}
\bigg( \int \abs[\Big]{\abs{x_2-x_1}-\abs{y_2-y_1}}^p\kappa(dx,d y)\bigg)^{\frac{1}{p}}&\leq \bigg( \int \Big(\abs{x_1-y_1}+\abs{x_2-y_2}\Big)^p\kappa(dx,d y)\bigg)^{\frac{1}{p}}\\
& \overset{\mathclap{\eqref{eq:defwp}}}{\le}  W_p(\pi,\tilde{\pi}). 
\end{split}\label{eq:term2Wp}
\end{align}
 Finally, plugging~\eqref{eq:term1L} and~\eqref{eq:term2Wp} into~\eqref{eq:holderLWp} completes the proof.
\end{proof}

\begin{proof}[Proof of Lemma~\ref{lem: third term L}]
According to Lemma~\ref{lem: AL}, the scaled cost~$\frac{c}{\veps}$ satisfies~\eqref{eq: AL} with
\begin{equation*}
L=\frac{C}{\veps}\Big[ M_p(\mu)+M_p(\nu)+M_p(\mu^{B_r})+M_p(\nu^{B_s}) \Big]^{p-1},
\end{equation*}
where~$C$ is a constant depending only on~$p$ and~$C_p$. It remains to bound~$M_p(\mu^{B_r})$ and~$M_p(\nu^{B_s})$. For this we note that
\begin{align}
\begin{split}
M_p(\mu)^p&=\int_{B_r}|x|^p\, \mu(dx) +\int_{B_r^c}|x|^p\, \mu(dx)\\
&=\frac{1}{\mu(B_r)}\int_{B_r}|x|^p\, \mu(dx)+\int_{B_r^c}|x|^p\, \mu(dx)-\frac{\mu(B_r^c)}{\mu(B_r)}\int_{B_r}|x|^p\, \mu(dx)\\
&\geq M_p(\mu^{B_r})^p+\int_{B_r^c}|r|^p\, \mu(dx)-\frac{\mu(B_r^c)}{\mu(B_r)}\int_{B_r}|r|^p\, \mu(dx)\\
&=M_p(\mu^{B_r})^p+|r|^p\mu(B_r^c)-\frac{\mu(B_r^c)}{\mu(B_r)}|r|^p\mu(B_r)=M_p(\mu^{B_r})^p. 
\end{split}
\label{eq: restrict second moment less than original}
\end{align}
An analogous argument holds for~$\nu$ and~$\nu^{B_s}$. Thus~\eqref{eq:LmuBnuB} follows.
\iffalse
Therefore,
by Lemma~\ref{lem: PsiAlphaIntegrals},
\begin{equation*}
M_p(\mu^{B_r})\leq M_p(\mu)\leq \Big[\frac{2p}{\alpha_{\mu}}\Gamma\Big(\frac{p}{\alpha_{\mu}}\Big)\Big]^{\frac1p}c_{\mu}^{-\frac{1}{\alpha_{\mu}}},\quad M_p(\nu^{B_s})\leq M_p(\nu)\leq \Big[\frac{2p}{\alpha_{\nu}}\Gamma\Big(\frac{p}{\alpha_{\nu}}\Big)\Big]^{\frac1p}c_{\nu}^{-\frac{1}{\alpha_{\nu}}},
\end{equation*}
which implies~\eqref{eq:LmuBnuB}.
\fi
\end{proof}

\begin{proof}[Proof of Lemma~\ref{lem: w2 of true approx}]
We bound~$W_p(\mu, \mu^{B_r})$ by constructing a coupling~$\hat\pi\in\Pi(\mu^{B_r},\mu)$ via
\begin{equation*}
\hat\pi := \mu(B_r) (x,x)_{\#}\mu^{B_r} + \mu(B_r^c) \Big( \mu^{B_r}\otimes \frac{\mu|_{B_r^c}}{\mu(B_r^c)}\Big),
\end{equation*}
where~$(x,x)_{\#}\mu^{B_r}$ denotes the push-forward measure of~$\mu^{B_r}$ through the map~$x\mapsto (x,x)$,~$\otimes$ denotes the product measure and~$\mu|_{B_r^C}$ is the restriction of~$\mu$ to~$B_r^c$. We estimate
\begin{align}
\begin{split}
W_p(\mu, \mu^{B_r})^p&\leq \int |x-y|^p\, \hat\pi (dx,dy)\\
&=\mu(B_r^c) \int |x-y|^p\, \Big( \mu^{B_r}\otimes \frac{\mu|_{B_r^c}}{\mu(B_r^c)}\Big)(dx,dy)\\
&\leq\mu(B_r^c) \int 2^{p-1}(|x|^p+|y|^p)\, \Big( \mu^{B_r}\otimes \frac{\mu|_{B_r^c}}{\mu(B_r^c)}\Big)(dx,dy)\\
&=2^{p-1}\mu(B_r^c)\paren[\big]{M_p(\mu^{B_r})^p+M_p(\mu^{B_r^c})^p}.
\end{split}
\label{eq:WpmuB}
\end{align} 
Analogously we have
\begin{equation}
W_p(\nu, \nu^{B_s})^p\leq  2^{p-1}\nu(B_s^c)\paren[\big]{M_p(\nu^{B_s})^p+M_p(\nu^{B_s^c})^p}.
\label{eq:WpnuB}
\end{equation}
Plugging~\eqref{eq:WpmuB} and~\eqref{eq:WpnuB} into~$W_p(\mu,\nu;\mu^{B_r},\nu^{B_s})=(W_p(\mu,\mu^{B_r})^p+W_p(\nu,\nu^{B_s})^p)^{1/p}$ finishes the proof.
\end{proof}

\begin{proof}[Proof of Lemma~\ref{lem: estimate second term-before take the expetation}]
We have
\begin{align}
\begin{split}
\abs{\mathcal{C}_{\veps}(\mu_n,\nu_n)-\mathcal{C}_{\veps}(\mu_n^{B_r},\nu_n^{B_s})} &\overset{\eqref{eq: scaling with epsilon}}{=} \veps\left|\mathcal{C}_{1}\Big(\mu_n,\nu_n,\frac{c}{\veps}\Big)-\mathcal{C}_{1}\Big(\mu_n^{B_r},\nu_n^{B_s},\frac{c}{\veps}\Big)\right|\\
&\overset{\mathclap{\eqref{eq:EOTstablility}}}{\leq}  L [W_p(\mu_n,\mu_n^{B})^p +W_p(\nu,\nu_n^{B})^p]^{1/p}, 
\end{split}
\label{eq:lem29proof}
\end{align}
where
\begin{equation*}
L=C\Big[M_p(\mu_n)+M_p(\nu_n)+M_p(\mu_n^{B_r})+M_p(\nu_n^{B_s})\Big]^{p-1}
\end{equation*}
from~\eqref{eq:defL} in Lemma~\ref{lem: AL}, and~$C$ only depends on~$p$ and~$C_p$. It remains to compute~$W_p(\mu_n, \mu_n^{B_r})$ and~$W_p(\nu_n, \nu_n^{B_s})$. We first compute~$W_p(\mu, \mu^{B_r})$. Using the coupling~$\hat\pi\in\Pi(\mu_n^{B_r},\mu_n)$ defined as 
\begin{equation*}
\hat\pi := \frac{n_r}{n} (x,x)_{\#} \mu_n^{B_r} + \Big(1-\frac{n_r}{n}\Big) \Big( \mu_n^{B_r}\otimes \frac{\mu_n|_{B_r^c}}{1-\frac{n_r}{n}}\Big)
\end{equation*}
similarly to the proof of Lemma~\ref{lem: w2 of true approx}, we bound
\begin{align}
\begin{split}
W_p(\mu_n, \mu_n^{B_r})^p&\leq \int |x-y|^p\, \hat\pi(dx,dy)\\
&=\Big(1-\frac{n_r}{n}\Big)\int |x-y|^p\,  \Big( \mu_n^{B_r}\otimes \frac{\mu_n|_{B_r^c}}{1-\frac{n_r}{n}}\Big)(dx,dy)\\
&\le 2^{p-1}\Big(1-\frac{n_r}{n}\Big)\frac{1}{n_r} \sum_{X_i\in B_r}|X_i|^p+2^{p-1}\frac{1}{n} \sum_{X_i\notin B_r}|X_i|^p\\
&= 2^{p-1}\Big(\frac{1}{n_r}-\frac{1}{n}\Big) \sum_{X_i\in B_r}|X_i|^p+2^{p-1}\frac{1}{n} \sum_{X_i\notin B_r}|X_i|^p.\
\end{split}
\label{eq:WpmuNmuBN}
\end{align}
Analogously, we obtain
\begin{equation}
W_p(\nu_n, \nu_n^{B_s})^p\leq 2^{p-1}\Big(\frac{1}{n_s}-\frac{1}{n}\Big)\sum_{Y_i\in B_s}|Y_i|^p+2^{p-1}\frac{1}{n} \sum_{X_i\notin B_s}|Y_i|^p.\label{eq:WpnuNnuBN}
\end{equation}
Plugging~\eqref{eq:WpmuNmuBN} and~\eqref{eq:WpnuNnuBN} into~\eqref{eq:lem29proof} completes the proof. 
\end{proof}

\begin{proof}[Proof of Lemma~\ref{lem: estimate of second term-after expectation}]

\restartsteps

\step Observe that Lemma~\ref{lem: estimate second term-before take the expetation} and H\"older's inequality yield
\begin{equation}\label{eq: A1A2}
\E\big[|\mathcal{C}_{\veps}(\mu_n,\nu_n)-\mathcal{C}_{\veps}(\mu_n^{B_r},\nu_n^{B_s})|\,\big| n_r,n_s\big]
\leq  C(A_1)^{\frac{p-1}{p}}\cdot (A_2)^{\frac1p},
\end{equation}
where
\begin{equation*}
\begin{split}
A_1&:=\E\big[\big(M_p(\mu_n)+M_p(\nu_n)+M_p(\mu_n^{B_r})+M_p(\nu_n^{B_s})\big)^p\,\big| n_r,n_s\big]\\
A_2&:=\E\Big[\Big(\frac{1}{n_r}-\frac{1}{n}\Big)\sum_{X_i\in B_r} |X_i|^p+\frac{1}{n} \sum_{X_i\notin B_r}|X_i|^p\\
&\qquad\quad +\Big(\frac{1}{n_s}-\frac{1}{n}\Big)\sum_{Y_i\in B_s}|Y_i|^p+\frac{1}{n}\sum_{X_i\notin B_s} |Y_i|^p\,\Big|n_r,n_s\Big].
\end{split}
\end{equation*}
It thus suffices to bound~$A_1$ and~$A_2$ respectively.
\step[Bounding~$A_1$] By the Cauchy-Schwarz inequality,
\begin{equation*}
\begin{split}
A_1&\leq 4^{p-1}\E\Big[\int|x|^p\,\mu_n(dx)+\int|y|^p\,\nu_n(dy)+\int|x|^p\,\mu_n^{B_r}(dx)+\int|y|^p\,\nu_n^{B_s}(dy)\,\Big|n_r,n_s\Big]\\
&=4^{p-1}\Big(\E\Big[\int|x|^p\,\mu_n(dx)\Big]+\E\Big[\int|y|^p\, \nu_n(dy)\Big]+\E\Big[\int|x|^p\, \mu_n^{B_r}(dx)\,\Big|n_r\Big]\\
&\quad +\E\Big[\int|y|^p\, \nu_n^{B_s}(dy)\,\Big|n_s\Big]\Big).
\end{split}
\end{equation*}
Since~$X_i\sim\mu$ we obtain 
\begin{equation}\label{eq:munMoment}
\E\Big[\int|x|^p\, d \mu_n\Big]=M_p(\mu)^p.
\end{equation}
For the restricted empirical measures we have by Remark~\ref{rem:binomial}
\begin{equation}\label{eq:munBMoment}
\E\Big[\int|x|^p\,\mu_n^{B_r}(dx)\,\Big| n_r\Big]
=\int |x|^p\, \mu^{B_r}(dx)\overset{\eqref{eq: restrict second moment less than original}}{\leq}M_p(\mu)^p. 
\end{equation}
We bound the other two terms in~$A_1$ in the same way. We thus obtain
\begin{equation}\label{eq: est A1}
A_1\leq C\paren{ M_p(\mu)^p+M_p(\nu)^p},
\end{equation}
where~$C$ only depends on~$p$.

\step[Bounding~$A_2$] By linearity,
\begin{equation*}
\begin{split}
A_2&= \E\Big[\Big(\frac{1}{n_r}-\frac{1}{n}\Big) \sum_{X_i\in B_r}|X_i|^p\,\Big|n_r\Big]+\E\Big[\sum_{X_i\notin B_r}\frac{1}{n}|X_i|^p \,\Big| n_r\Big]\\
&\quad +\E\Big[ \Big(\frac{1}{n_s}-\frac{1}{n}\Big) \sum_{Y_i\in B_s} |Y_i|^p \,\Big| n_s\Big] + \E\Big[ \frac{1}{n} \sum_{X_i\notin B_s} |Y_i|^p\,\Big| n_s\Big].
\end{split}
\end{equation*}
We bound the first two terms. For this we note that, using again Remark~\ref{rem:binomial},
\begin{equation}\label{eq:EXrnr}
\E\Big[\Big(\frac{1}{n_r}-\frac{1}{n}\Big) \sum_{X_i\in B_r} |X_i|^p\,\Big|\: n_r\Big]
=\Big(1-\frac{n_r}{n}\Big) M_p(\mu_{B_r})^p\overset{\eqref{eq: restrict second moment less than original}}{\leq}\Big(1-\frac{n_r}{n}\Big)M_p(\mu)^p. 
\end{equation}
Similarly,
\begin{equation}
\E\Big[\frac{1}{n} \sum_{X_i\notin B_r}|X_i|^p\,\Big| n_r\Big]=\E\Big[ \frac{n-n_r}{n}  \frac{1}{n-n_r}\sum_{X_i\notin B_r} |X_i|^p \,\Big| n_r\Big]=\Big(1-\frac{n_r}{n}\Big) M_p(\mu^{B_r^c})^p.\label{eq:EXrcnr}
\end{equation}
Analogously,
\begin{equation*}
\begin{split}
\E\Big[\Big(\frac{1}{n_s}-\frac{1}{n}\Big) \sum_{Y_i\in B_s} \frac{1}{n_s}|Y_i|^p\,\Big|n_s\Big]&\leq \Big(1-\frac{n_s}{n}\Big) M_p(\nu)^p,\\
\E\Big[\frac{1}{n} \sum_{Y_i\notin B_s}|Y_i|^p]&= \Big(1-\frac{n_s}{n}\Big) M_p(\nu^{B_s^c})^p.
\end{split}
\end{equation*}
Therefore,
\begin{equation}\label{eq: est A2}
A_2\leq \paren[\Big]{M_p(\mu)^p+M_p(\nu^{B_s^c})^p}\cdot  \Big(1-\frac{n_r}{n}\Big) 
+\paren[\Big]{M_p(\nu)^p+M_p(\nu^{B_s^c})^p}\cdot\Big(1-\frac{n_s}{n}\Big).
\end{equation}
Plugging~\eqref{eq: est A1} and~\eqref{eq: est A2} into~\eqref{eq: A1A2} finishes the proof.
\end{proof}

\subsection{Remaining proofs from Section~\ref{sec:stromme}}

\begin{proof}[Proof of Lemma~\ref{lem: estimate first term-before taking expectation}]
Lemma~\ref{lem: estimate first term-before taking expectation} follows from an analog of \cite[Section 5.1]{stromme2023minimum}. Indeed, the only property used there is the  concavity of the ROT dual function 
$\Phi_{\veps}^{PQ}: L^{\infty}(P)\times L^{\infty}(Q)\to\mathbb{R}$ defined as
\begin{equation}
 \Phi_{\veps}^{PQ}(f,g):= \int f \, d P+ \int g\, d Q -\veps \int \psi\Big( \frac{f+g-c}{\veps}\Big)\, d (P\otimes Q)
\end{equation}
for compactly supported probability measures~$P,Q$ on~$\mathbb{R}^d$; see \cite[Proposition 13]{stromme2023minimum}. We establish this property in Corollary~\ref{cor:marginalRounding} below. It remains to follow \cite[Section 5.1]{stromme2023minimum} line by line to conclude the proof.
\end{proof}

\begin{proposition}\label{prop:concav} Let~$P,Q$ be probability measures on~$\R^d$. For pairs~$(f_0,g_0), (f_1,g_1),  \in L^{\infty}(P)\times L^{\infty}(Q)$ we have 
\begin{align}
\begin{split}
\MoveEqLeft\Phi_{\veps}^{PQ}(f_0,g_0)-\Phi_{\veps}^{PQ}(f_1,g_1)\\
&\le \int \paren[\big]{f_0(x)-f_1(x)}\paren[\Big]{1-\int\psi'\paren[\Big]{\frac{f_1(x)+g_1(y)-c(x,y)}{\veps}}\, Q(dy)}\, P(dx)\\
&+\int \paren[\big]{g_0(y)-g_1(y)}\paren[\Big]{1-\int\psi'\paren[\Big]{\frac{f_1(x)+g_1(y)-c(x,y)}{\veps}}\, P(dx)}\, Q(dy), \label{eq:concav1}
\end{split}
\end{align}
and
\begin{align}
\begin{split}
\MoveEqLeft\Phi_{\veps}^{PQ}(f_0,g_0)-\Phi_{\veps}^{PQ}(f_1,g_1)\\
&\ge\int \paren[\big]{f_0(x)-f_1(x)}\paren[\Big]{1-\int\psi'\paren[\Big]{\frac{f_0(x)+g_0(y)-c(x,y)}{\veps}}\, Q(dy)}\, P(dx)\\
&+\int \paren[\big]{g_0(y)-g_1(y)}\paren[\Big]{1-\int\psi'\paren[\Big]{\frac{f_0(x)+g_0(y)-c(x,y)}{\veps}}\, P(dx)}\, Q(dy). \label{eq:concav2}
\end{split}
\end{align}
\end{proposition}

\begin{proof}[Proof of Proposition~\ref{prop:concav}]
By Lemma~\ref{lem:duality} we have
\begin{align}
 \Phi_{\veps}^{PQ}(f_0,g_0)-\Phi_{\veps}^{PQ}(f_1,g_1)&=\int (f_0-f_1)\, d P +\int (g_0-g_1)\, d Q
 \\
 &-\veps \int \psi\Big( \frac{f_0+g_0-c}{\veps}\Big)-\psi\Big( \frac{f_1+g_1-c}{\veps}\Big)\, d(P\otimes Q).
\end{align}
As~$\psi$ is convex, it can be directly checked that
\begin{equation}\label{eq:cvxPsi1}
 \psi\Big( \frac{f_0+g_0-c}{\veps}\Big)-\psi\paren[\Big]{ \frac{f_1+g_1-c}{\veps}}\geq \psi'\paren[\Big]{\frac{f_1+g_1-c}{\veps}}\cdot \paren[\Big]{\frac{(f_0+g_0)-(f_1+g_1)}{\veps}}
\end{equation}
and
\begin{equation}\label{eq:cvxPsi2}
 \psi\Big( \frac{f_0+g_0-c}{\veps}\Big)-\psi\paren[\Big]{ \frac{f_1+g_1-c}{\veps}}\leq \psi'\paren[\Big]{ \frac{f_0+g_0-c}{\veps}}\cdot \paren[\Big]{\frac{(f_0+g_0)-(f_1+g_1)}{\veps}}.
\end{equation}
Then a direct computation yields~\eqref{eq:concav1} and~\eqref{eq:concav2}.
\end{proof}

\begin{corollary}\label{cor:marginalRounding}
Let~$f,g\in L^{\infty}(P)\times L^{\infty}(Q)$, where~$P, Q$ have compact support~$\Omega,\Omega'$ respectively. Let~$\tilde{f}$ satisfy
\begin{equation}\label{eq:defTildeF}
\int\psi'\paren[\Big]{\frac{\tilde{f}(x)+g(y)-c(x,y)}{\veps}}\, Q(dy)=1\quad \text{ for all }x\in \Omega.
\end{equation}
Then
\begin{equation}\label{eq:marginalRoundingF}
    \Phi_{\veps}^{PQ}(f,g)\le \Phi_{\veps}^{PQ}(\tilde{f},g).
\end{equation}
An analogous statement holds for~$g$.
\end{corollary}
\begin{proof}[Proof of Corollary~\ref{cor:marginalRounding}] We only prove~\eqref{eq:marginalRoundingF}. The existence of such~$\tilde{f}$ is guaranteed by Lemma \ref{lem:C1GEN25}.
    Notice that by concavity, we have
    \begin{align*}
         \MoveEqLeft \Phi_{\veps}^{PQ}(f,g)- \Phi_{\veps}^{PQ}(\tilde{f},g)\\
         &\overset{\mathclap{\eqref{eq:concav1}}}{\leq} \int \paren[\big]{f(x)-\tilde{f}(x)}\paren[\Big]{1-\int\psi'\paren[\Big]{\frac{\tilde{f}(x)+g(y)-c(x,y)}{\veps}}\, Q(dy)}\, P(dx) \overset{\eqref{eq:defTildeF}}{=}0.
    \end{align*}
\end{proof}

\begin{proof}[Proof of Lemma~\ref{lem:L2normP}]

For simplicity of notation, we define the following subset~$U\subseteq B^{r}\times B^{s}$ as
\begin{equation}\label{eq:defU}
    U:= \left\{(x,y)\in B^{r}\times B^{s}:   \psi'\paren[\Big]{\frac{f^{r,s}(x)+g^{r,s}(y)-c(x,y)}{\veps} } \leq 1\right\}.
\end{equation}
We then compute 
\begin{align}
\norm{p^{r,s}(x,y)}^2_{L^2(\mu^{r,s}\otimes \nu^{r,s})}
&=\int_U \paren{p^{r,s}(x,y)}^2\,\mu^{B_r}(dx)\,\nu^{B_s}(dy)+\int_{U^c}\paren{p^{r,s}(x,y)}^2\,\mu^{B_r}(dx)\,\nu^{B_s}(dy)\\
     &\overset{\mathclap{\eqref{eq:defU}}}{\leq}1+\int_{U^c}\paren{p^{r,s}(x,y)}^2\,\mu^{B_r}(dx)\,\nu^{B_s}(dy). \label{eq:pSeq}
\end{align}
Next we bound the second term on the right-hand side of~\eqref{eq:pSeq}. We observe that
\begin{equation}\label{eq:defUc}
    U^c=\left\{(x,y)\in B_{r}\times B_{s}:   \frac{f^{r,s}(x)+g^{r,s}(y)-c(x,y)}{\veps} > t_0\right\}
\end{equation}
by Assumption~\ref{assum:psi}.
For~$(x,y)\in U^c$, for every~$y'$ such that
\begin{equation}\label{eq:defEta}
    \abs{y-y'}\leq \frac{\tilde{\delta}\veps}{2C_p(r+s)^{p-1}}=: \eta,
\end{equation}
we have according to~\eqref{eq:Lipc} and Lemma~\ref{lem: Lip potential}
\begin{equation}\label{eq:NeighG}
    \frac{\abs{g^{r,s}(y)-g^{r,s}(y')}+\abs{c(x,y)-c(x,y')}}{\veps}\leq \tilde{\delta}. 
\end{equation}
Thus
\begin{align}
   \MoveEqLeft \frac{f^{r,s}(x)+g^{r,s}(y')-c(x,y')}{\veps}\\
    &\geq  \frac{f^{r,s}(x)+g^{r,s}(y)-c(x,y)}{\veps}-\frac{\abs{g^{r,s}(y)-g^{r,s}(y')}+\abs{c(x,y)-c(x,y')}}{\veps}\\
    &\overset{\mathclap{\eqref{eq:NeighG}}}{\geq} \frac{f^{r,s}(x)+g^{r,s}(y)-c(x,y)}{\veps} - \tilde{\delta}
    \overset{\eqref{eq:delta},\eqref{eq:defUc}}{\geq} \frac{t_0}{2}. \label{eq:yNeighLB}
\end{align}
Before we proceed with the proof, we first state a calculus fact. For~$a,b\in [t_0-\tilde{\delta},+\infty)$
\begin{align}
\begin{split}
\frac{\psi'(a)}{\psi'(b)}&=\exp\paren[\Big]{\log \psi'(a) -\log \psi'(b)}\geq \exp\paren[\Big]{-\frac{\abs{\psi'(b)-\psi'(a)}}{\max\{\psi'(b),\psi'(a)\}}}\\
&\geq \exp\paren[\Big]{-\max_{c\in [a,b]}\abs[\Big]{\frac{\psi''(c)}{\psi'(c)}}\abs{b-a}} \overset{\eqref{eq:psi}}{\geq} \exp\paren[\Big]{-\frac{C_{\psi}}{\min\{|a|^{\gamma},|b|^{\gamma}\}}|b-a|}. \label{eq:DpsiLB}
\end{split}
\end{align}
The above fact implies that for~$(x,y)\in U^c$, and~$y'$ such that~$|y-y'|\leq\eta(y)$, we can plug in~$a=f^{r,s}(x)+g^{r,s}(y')-c(x,y')$,~$b=f^{r,s}(x)+g^{r,s}(y)-c(x,y)$ so that
\begin{align}
\frac{p^{r,s}(x,y')}{p^{r,s}(x,y)}\quad&\overset{\mathclap{\eqref{eq:yNeighLB},~\eqref{eq:DpsiLB}}}{\geq} \quad\exp\paren[\Big]{-\frac{C_{\psi}}{\veps(t_0/2)^{\gamma}}\paren[\big]{\abs{g^{r,s}(y)-g^{r,s}(y')}+\abs{c(x,y)-c(x,y')}}}\\
&\overset{\mathclap{\eqref{eq:NeighG}}}{\geq} \quad\exp\paren[\Big]{-\frac{C_{\psi}\tilde{\delta}}{(t_0/2)^{\gamma}}}. \label{eq:RatioP}
\end{align}
Thus
\begin{align}\label{eq:split}
\begin{split}
1&=p^{r,s}(x,y)\int_{B_s} \frac{p^{r,s}(x,y')}{p^{r,s}(x,y)}\, \nu^{B_s}(dy')\geq p^{r,s}(x,y)\int_{B(y,\eta)} \frac{p^{r,s}(x,y')}{p^{r,s}(x,y)}\, \nu^{B_s}(dy')\\
&\overset{\mathclap{\eqref{eq:DpsiLB}}}{\ge} p^{r,s}(x,y)\int_{B(y,\eta)} \exp\paren[\Big]{-\frac{C_{\psi}\tilde{\delta}}{(t_0/2)^r}}\, \nu^{B_s}(dy')\\
&\overset{\mathclap{\eqref{eq:defEta}}}{=} p^{r,s}(x,y)\cdot\nu^{B_s}\Big( B\Big(y,\frac{\tilde{\delta}\veps}{2C_p(r+s)^{p-1}}\Big)\Big)e^{-\frac{C_{\psi}\tilde{\delta}}{(t_0/2)^{\gamma}}}.
\end{split}
\end{align}
Therefore,
\begin{align*}
\MoveEqLeft\int_{U^c} [p^{r,s}(x,y)]^2 \mu^{B_r}(dx)\nu^{B_s}(dy)\\
&\overset{\mathclap{\eqref{eq:split}}}{\leq} e^{\frac{C_{\psi}\tilde{\delta}}{(t_0/2)^{\gamma}}}  \int_{U^c} \nu^{B_s}\left( B\Big(y,\frac{\tilde{\delta}\veps}{2C_p(r+s)^{p-1}}\Big)\right)^{-1} p^{r,s}(x,y)\, \mu^{B_r}(dx)\,\nu^{B_s}(dy)\\
&\leq e^{\frac{C_{\psi}\tilde{\delta}}{(t_0/2)^{\gamma}}}  \int \nu^{B_s}\left( B\Big(y,\frac{\tilde{\delta}\veps}{2C_p(r+s)^{p-1}}\Big)\right)^{-1} p^{r,s}(x,y)\, \mu^{B_r}(dx)\,\nu^{B_s}(dy).
\end{align*}
Applying Lemma~\ref{lem: average inverse mass} we can further bound
\begin{align*}
\MoveEqLeft\int \nu^{B_s}\left( B\Big(y,\frac{\tilde{\delta}\veps}{2C_p(r+s)^{p-1}}\Big)\right)^{-1} p^{r,s}(x,y)\, \mu^{B_r}(dx)\,\nu^{B_s}(dy)\\
&= \int \nu^{B_s}\left( B\Big(y,\frac{\tilde{\delta}\veps}{2C_p(r+s)^{p-1}}\Big)\right)^{-1}\, \nu^{B_s}(dy)
\leq \mathcal{N} \Big(B^\nu_s, \frac{\tilde{\delta}\veps}{2C_p(r+s)^{p-1}}\Big).
\end{align*}
By an analogous argument,
\begin{equation*}
\int_{U^c} [p^{r,s}(x,y)]^2\,  \mu^{B_r}(dx)\,\nu^{B_s}(dy) 
\leq e^{\frac{C_{\psi}\tilde{\delta}}{(t_0/2)^{\gamma}}}\mathcal{N} \Big(B^\mu_r, \frac{\tilde{\delta}\veps}{2C_p(r+s)^{p-1}}\Big).
\end{equation*}
We can thus conclude that
\begin{align}
\MoveEqLeft\int_{U^c} [p^{r,s}(x,y)]^2\,  \mu^{B_r}(dx)\,\nu^{B_s}(dy)\\ 
&\leq e^{\frac{C_{\psi}\tilde{\delta}}{(t_0/2)^{\gamma}}}\paren[\bigg]{\mathcal{N} \Big(B^\mu_r, \frac{\tilde{\delta}\veps}{2C_p(r+s)^{p-1}}\Big)\wedge \mathcal{N}\Big(B^\nu_s, \frac{\tilde{\delta}\veps}{2C_p(r+s)^{p-1}}\Big)}. \label{eq:L2pUc}
\end{align}
Plugging~\eqref{eq:L2pUc} into~\eqref{eq:pSeq} and then taking the square-root complete the proof.
\end{proof}

\begin{proof}[Proof of Corollary~\ref{cor: first term-before taking expecation}]
Clearly,
\begin{equation*}
\begin{split}
& \E\big[ \|(f^{r,s}_n-f^{r,s},g^{r,s}_n-g^{r,s})\|^2_{L^2(\mu_n^{B_r})\times L^2(\nu_n^{B_s})}\big| n_r, n_s \big]\\
&=    \E\Big[ \int |f^{r,s}_n(x)-f^{r,s}(x)|^2\, \mu_n^{B_r}(dx)+\int |g^{r,s}_n(y)-g^{r,s}(y)|^2\, \nu_n^{B_s}(dy) \Big| n_r, n_s  \Big]\\
&\leq 2\E\Big[ \int (|f^{r,s}_n(x)|^2+|f^{r,s}(x)|^2)\, \mu_n^{B_r}(dx)+\int (|g^{r,s}_n(y)|^2+|g^{r,s}(y)|^2)\, \nu_n^{B_s}(dy)\Big| n_r, n_s  \Big]\\
&\leq 2\left[ \|f^{r,s}\|^2_{L^{\infty}(\mu^{B_r})}+ \|g^{r,s}\|^2_{L^{\infty}(\nu^{B_s})}+ \|f^{r,s}_n\|^2_{L^{\infty}(\mu^{B_r})}+ \|g^{r,s}_n\|^2_{L^{\infty}(\nu^{B_s})} \big| n_r, n_s  \right]\\
&\leq 8\Big(C_p(r+s)^p\Big)^2,
\end{split}
\end{equation*}
where we used Lemma~\ref{lem:LinftyBoundSchrondinger} for the last inequality.
By Lemma~\ref{lem:varfg}
\begin{equation*}
\begin{split}
     &\Var_{\mu^{B_r}}(f^{r,s})\leq \|f^{r,s}\|^2_{L^{\infty}(\mu^{B_r})}\leq \Big(C_p(r+s)^p\Big)^2\\
     &\Var_{\nu^{B_s}}(g^{r,s})\leq \|g^{r,s}\|^2_{L^{\infty}(\nu^{B_s})}\leq \Big(C_p(r+s)^p\Big)^2.
\end{split}
\end{equation*}
Combining the above estimates with Lemma~\ref{lem:L2normP} and plugging them into~\eqref{eq:ROTCpt} finishes the proof.
\end{proof}

\subsection{Remaining proofs from Section~\ref{sec:proofmain}}

\begin{proof}[Proof of Lemma~\ref{lem: nx>0, or ny>0}]
We first consider the case~$n_r=0=n_s$.
Plugging~$\pi=\mu\otimes \nu$ into~\eqref{eq: ROT} yields
\begin{align}
\begin{split}
\mathcal{C}_\epsilon(\mu,\nu)&\le \int c(x,y)\,\mu(dx)\nu(dy)\overset{\text{Lem. }~\ref{lem: pointwise bound of cost function}}{\le} C_p\int |x-y|^p\,\mu(dx)\nu(dy)\\
&\le  2^{p-1}C_p\Big(M_p(\mu)^p+M_p(\nu)^p\Big). \label{eq:Cmunu}
\end{split}
\end{align}
Similarly,
\begin{equation}\label{eq:Cmunuhat}
    \mathcal{C}_\epsilon(\mu_n,\nu_n)\le 2^{p-1}C_p\Big(M_p(\mu_{n})^p+M_p(\nu_{n})^p\Big).
\end{equation}
On the event~$\{n_r=0,n_s=0\}$ we clearly have
$\mu_n=\mu_n^{B_r^c}$ and~$ \nu_n= \nu_n^{B_s^c}$.
By the triangle inequality we conclude
\begin{equation}\label{eq:CmuNxNy0}
|\mathcal{C}_{\veps}(\mu,\nu)-\mathcal{C}_{\veps}(\mu_n,\nu_n)|
\leq  C \big[ 1+M_p(\mu)^p+M_p(\nu)^p+ M_p(\mu_n^{B_r^c})^p + M_p(\nu_n^{B_s^c})^p \big].
\end{equation}
Taking conditional expectations on both sides of~\eqref{eq:CmuNxNy0} finishes the proof for$n_r=0=n_s$.

Next, we only prove~\eqref{eq:nx>0} as~\eqref{eq:ny>0} follows from a symmetric argument. Following the same steps as above with~$\mu_n^{B_r^c}$ replaced by~$\mu_n$ we obtain 
\begin{equation}\label{eq:CmuNx>0}
|\mathcal{C}_{\veps}(\mu,\nu)-\mathcal{C}_{\veps}(\mu_n,\nu_n)|
\leq  C \big[ M_p(\mu)^p+M_p(\nu)^p+ M_p(\mu_n)^p + M_p(\nu_n^{B_s^c})^p \big].
\end{equation}
We note that, by Remark~\ref{rem:binomial},
\begin{align*}
    \E \big[M_p(\mu_n)^p\,|\, n_r=i,n_s=0\big]&=\frac{1}{n}\E \Big[\sum_{X_j\in B_r}|X_j|^p + \sum_{X_j\notin B_r}|X_j|^p\,|\, n_r=i\Big]\\
    &=\frac{1}{n}\big[ i M_p(\mu^{B_r})^p+ (n-i)M_p(\mu^{B_r^c})^p \big]\\
    &\overset{\mathclap{\eqref{eq: restrict second moment less than original}}}{\leq} \frac{i}{n}M_p(\mu)^p + (1-\frac{i}{n})M_p(\mu^{B_r^c})^p.
\end{align*}
Taking conditional expectations on both sides of~\eqref{eq:CmuNx>0} finishes the proof.
\end{proof}

\begin{proof}[Proof of Lemma~\ref{lem: expectation of empirical approx}]
According to Lemma~\ref{lem: estimate of second term-after expectation},
\begin{equation}\label{eq: nx,ny>0, step1}
\sum_{i,j=1}^{n}\E \big[ |\mathcal{C}_{\veps}(\mu_n^{B_r},\nu_n^{B_s})-\mathcal{C}_{\veps}(\mu_n,\nu_n)|\,\big| n_r=i,n_s=j\big] \cdot \mathbb{P}(n_r=i,n_s=j)
\le   C\Big(M_p(\mu)^p+M_p(\nu)^p\Big)^{\frac{p-1}{p}}\cdot I_1
\end{equation}
where
\begin{equation}
\begin{split}
I_1&:= \sum_{i,j=1}^{n}\bigg[\paren[\Big]{M_p(\mu)^p+M_p(\mu^{B_r^c})^p}\cdot  \Big(1-\frac{i}{n}\Big)\\
&\qquad\qquad  +\paren[\Big]{M_p(\nu)^p+M_p(\nu^{B_s^c})^p}\cdot\Big(1-\frac{j}{n}\Big)\bigg]^{\frac1p}\cdot \mathbb{P}(n_r=i, n_s=j).
    \end{split}
\end{equation}
By Jensen's inequality
\begin{equation}\label{eq: Jensen}
\begin{split}
(I_1)^p&\leq  
\paren[\Big]{M_p(\mu)^p+M_p(\mu^{B_r^c})^p}\cdot  \sum_{i,j=1}^{n} \Big(1-\frac{i}{n}\Big) \mathbb{P}(n_r=i, n_s=j)\\
&\qquad +\paren[\Big]{M_p(\nu)^p+M_p(\nu^{B_s^c})^p}\sum_{i,j=1}^{n} \Big(1-\frac{j}{n}\Big) \mathbb{P}(n_r=i, n_s=j). 
\end{split}
\end{equation}
Now we bound each term on the right-hand side of~\eqref{eq: Jensen}.
As~$n_r \sim \mathrm{Bin}(n, \mu(B_r)$ by Remark~\ref{rem:binomial}, we have
\begin{align}
\begin{split}
\sum_{i,j=1}^{n}\Big(1-\frac{i}{n}\Big)\mathbb{P}(n_r=i,n_s=j)
&\leq \sum_{i=0}^{n} \Big(1-\frac{i}{n} \Big) \mathbb{P}(n_r=i) 
= 1-\frac{n\mu(B_r)}{n} = \mu(B_r^c). \label{eq:E1-nr/n}
\end{split}
\end{align}
Analogously,
\begin{equation}\label{eq:E1-ns/n}
\sum_{i,j=1}^{n}\Big(1-\frac{j}{n}\Big) \mathbb{P}(n_r=i,n_s=j)
\leq \nu(B_s^c).
\end{equation}
Therefore,
\begin{equation}
\begin{split}
I_1^p\quad&\overset{\mathclap{\eqref{eq:E1-nr/n},\eqref{eq:E1-ns/n}}}{\leq}\quad \paren[\Big]{M_p(\mu)^p+M_p(\mu^{B_r^c})^p}\mu(B_r^c)+\paren[\Big]{M_p(\nu)^p+M_p(\nu^{B_s^c})^p}\nu(B_s^c). \label{eq:I1}
\end{split}
\end{equation}
Plugging~\eqref{eq:I1} into~\eqref{eq: nx,ny>0, step1} finishes the proof.
\end{proof}

\begin{proof}[Proof of Lemma~\ref{lem: choiceRxRy}]
Let us first remark that~\eqref{eq:IntGamma1} follows directly from Lemma~\ref{lem: PsiAlphaIntegrals} in the appendix. It thus remains to prove~\eqref{eq:GammaChooseRS} and~\eqref{eq:TailChooseRS}.

\restartsteps

\step[Bounding~$\mu((B_n^{\mu})^c), \nu((B_n^{\nu})^c)$] Observe that~\eqref{eq:chooseRxRy} implies 
\begin{equation}\label{eq:chooselog}
c_{\mu}\big(r_n^\mu\big)^{\alpha_{\mu}}\geq \log(n^{p}),\quad c_{\nu}\big(r_n^\nu\big)^{\alpha_{\nu}}\geq\log(n^{p}),
 \end{equation}
and thus
\begin{align*}
\exp\big(-c_{\mu}(r_n^\mu)^{\alpha_{\mu}}\big)
\leq n^{-p},\quad \exp\big(-c_{\nu}(r_n^\nu)^{\alpha_{\nu}}\big)\leq n^{-p}.
\end{align*}
Together with~\eqref{eq:PsiAlapha} this shows~\eqref{eq:TailChooseRS}.

\step[Bounding~$M_p(\mu^{(B_n^{\mu})^c})^p, M_p(\nu^{(B_n^{\nu})^c})^p$] We only prove the estimate of~$M_p(\mu^{(B_n^{\mu})^c})^p$, as the estimate of~$M_p(\nu^{(B_n^{\nu})^c})^p$ follows analogously. We also set~$r=r_n^\mu$ for notational simplicity.
Recalling that
\begin{equation}
M_p(\mu^{B_r^c})^p = \frac{1}{\mu(B_r^c)}\int_{B_r^c} |x|^p\, d \mu(x),
\end{equation}
Lemma~\ref{lem: PsiAlphaIntegrals} yields the bound 
\begin{align}\label{eq:write_explicit}
\mu(B_r^c) M_p(\mu^{B_r^c})^p\le  2r^p \exp\big(-c_{\mu}r^{\alpha_{\mu}}\big)+\frac{2p}{\alpha_{\mu}}c_{\mu}^{-\frac{p}{\alpha_{\mu}}}\Gamma\Big(\frac{p}{\alpha_{\mu}},c_{\mu}r^{\alpha_{\mu}}\Big).
\end{align}
We first bound~$r^p \exp\big(-c_{\mu}r^{\alpha_{\mu}}\big)$. For this we observe that~\eqref{eq:chooseRxRy} implies 
\begin{equation}
r^{\alpha_{\mu}}\geq \paren[\Big]{\frac{2p}{c_{\mu}\alpha_{\mu}}}^2.
\end{equation}
By Lemma~\ref{lem: calcfact} with~$x=r^{\alpha_{\mu}}$ and~$a=2p/(c_\mu\alpha_\mu)$ we have
\begin{align}
p\log r =\frac{p}{\alpha_{\mu}}\log(r^{\alpha_{\mu}})\leq \frac{c_{\mu}r^{\alpha_{\mu}}}{2},
\end{align}
which yields
\begin{equation}\label{eq:A51}
r^p \exp\big(-c_{\mu}r^{\alpha_{\mu}}\big)=\exp\paren[\big]{-c_{\mu}r^{\alpha_{\mu}}+p\log r}\leq \exp\paren[\big]{-\frac{c_{\mu}}{2}r^{\alpha_{\mu}}}\overset{\eqref{eq:chooselog}}{\leq} \frac{1}{n^\frac{p}{2}}.
\end{equation}
Next, we want to apply Lemma~\ref{l:estGamma} to estimate the second term. For this we observe that~\eqref{eq:chooseRxRy} implies 
\begin{equation}\label{eq:RxRyp2log}
c_{\mu}r^{\alpha_{\mu}}\geq \Big(\frac{p}{\alpha_{\mu}}\vee 1\Big)\log(n^{p}).
 \end{equation}
A direct calculation yields, that for all~$p\geq 1$, 
\begin{equation}\label{eq:logpre}
    4^{\frac1p}p^{\frac{2}{p}}=(2p)^{\frac{2}{p}}=\exp\paren[\Big]{4\frac{\log(2p)}{2p}}\leq \exp\paren[\Big]{\frac{4}{e}}\leq 5,
\end{equation}
where we use the fact that~$x\mapsto \log(x)/x$ achieves its maximum when~$x= e$.
Therefore, for all~$p\geq 1$,~$\alpha_{\mu}\geq 1$,
\begin{equation}
    4^{\frac{1}{p}}\Big(\frac{p}{\alpha_{\mu}} \Big)^{\frac{2}{p}}\overset{\eqref{eq:logpre}}{\le} 5\Big(\frac{1}{\alpha_{\mu}} \Big)^{\frac{2}{p}}\le 5\leq n
\end{equation}
which implies that 
\begin{align}\label{eq:condition}
n^{p}\geq  4\big(\frac{p}{\alpha_{\mu}} \big)^2.
\end{align}
By monotonicity of the incomplete Gamma function we obtain that 
\begin{equation}\label{eq:A52}
\Gamma\Big(\frac{p}{\alpha_{\mu}},c_{\mu}r^{\alpha_{\mu}}\Big)\overset{\eqref{eq:RxRyp2log}}{\leq} \Gamma\Big(\frac{p}{\alpha_{\mu}},\Big(\frac{p}{\alpha_{\mu}}\vee 1\Big)\log(n^{p})\Big)\le \frac{1}{n^{p}},
\end{equation}
where we used Lemma~\ref{l:estGamma} with~$s=p/a_\mu$ and~$x=n^p$ for the last inequality (recalling~\eqref{eq:condition}).
Plugging~\eqref{eq:A51} and~\eqref{eq:A52} into~\eqref{eq:write_explicit} finishes the proof.
\end{proof}

\bibliographystyle{amsalpha}
\bibliography{references}

\appendix
\section{Auxiliary Lemmas}\label{sec:appa}

\begin{lemma}\label{lem: PsiAlphaIntegrals}
Let~$p\geq 1$. If Assumption~\ref{assumption:Psialpha} holds,
then
\begin{equation}\label{eq:IntGamma}
M_p( \mu)^p\leq \frac{2p}{\alpha_{\mu}}c_{\mu}^{-\frac{p}{\alpha_{\mu}}}\Gamma\Big(\frac{p}{\alpha_{\mu}}\Big),
\end{equation}
and for~$r>0$
\begin{equation}\label{eq:IntGammaAlpha}
\int_{B_r^c} |x|^p\, d \mu(x)\leq 2r^p \exp\big(-c_{\mu}r^{\alpha_{\mu}}\big)+\frac{2p}{\alpha_{\mu}}c_{\mu}^{-\frac{p}{\alpha_{\mu}}}\Gamma\Big(\frac{p}{\alpha_{\mu}},c_{\mu}r^{\alpha_{\mu}}\Big),
\end{equation}
where~$\Gamma$ was defined in~\eqref{eq:gamma}.
\end{lemma}
\begin{proof}
By Fubini's theorem,
\begin{align*}
\int_{B_r^c} |x|^p\, d \mu(x) 
&= \int_0^\infty \mu(|x|\ge r, |x|\ge t^{1/p})\, dt\\
&= r^p \mu(B_{r}^c) +\int_{r^p}^\infty \mu(B_{t^{1/p}}^c\big)dt\\
&\overset{\eqref{eq:PsiAlapha}}{\leq} 2r^p \exp(-c_{\mu}r^{\alpha_{\mu}})  +\int_{r^p}^{\infty}2\exp\Big(-c_{\mu}t^{\frac{\alpha_{\mu}}{p}}\Big)\, d t\\
&\overset{\mathclap{z:=c_{\mu}t^{\frac{\alpha_{\mu}}{p}}}}{=} \quad 2r^p \exp(-c_{\mu}r^{\alpha_{\mu}})+ \int_{c_{\mu}r^{\alpha_{\mu}}}^{\infty} \frac{2p}{\alpha_{\mu}}c_{\mu}^{-\frac{p}{\alpha_{\mu}}}z^{\frac{p}{\alpha_{\mu}}-1}\exp(-z)\, d z\\
&=2r^p \exp(-c_{\mu}r^{\alpha_{\mu}})+\frac{2p}{\alpha_{\mu}}c_{\mu}^{-\frac{p}{\alpha_{\mu}}}\Gamma\Big(\frac{p}{\alpha_{\mu}},c_{\mu}r^{\alpha_{\mu}}\Big).
\end{align*}
In particular, if~$r=0,$
\[
\int|x|^{p}\, d\mu(x)\leq \frac{2p}{\alpha_{\mu}}c_{\mu}^{-\frac{p}{\alpha_{\mu}}}\Gamma\Big(\frac{p}{\alpha_{\mu}}\Big). \qedhere
\]
\end{proof}

\begin{lemma}\label{lem: binomial integrals}
Let~$n\geq 4$.  Let~$a\in (0,1)$ satisfy
\begin{equation}\label{eq:pbin}
    a\geq 1-\frac{2}{n^2}.
\end{equation}
Then there exists an absolute constant ~$C>0$ such that
\begin{equation}\label{eq:binOverN}
\begin{split}
&\sum_{j=1}^{n}j^{-\frac12}C_n^j a^j (1-a)^{n-j}
\leq \frac{C}{\sqrt{n}},\\
&\sum_{j=1}^{n}j^{-\frac14}C_n^j a^j (1-a)^{n-j}\leq \frac{C}{\sqrt[4]{n}},
\end{split}
\end{equation}
where~$C_n^j:= \binom{n}{j}$.
\end{lemma}

\begin{proof}
Note that for~$1\leq j\leq n-1$ we have
\begin{equation*}
\begin{split}
    &\frac{j^{-\frac12}C_n^{j} a^{j} (1-a)^{n-j}}{(j+1)^{-\frac12}C_n^{j+1} a^{j+1} (1-a)^{n-j-1}}
    =\frac{\sqrt{j+1}}{\sqrt{j}}\frac{j+1}{n-j}\frac{1-a}{a}.
\end{split}
\end{equation*}
Next, if~$n\geq 4$,~$1\leq j\leq n-1$ and~$a\geq 1-\frac{2}{n^2}$,
\[
\frac{\sqrt{j+1}}{\sqrt{j}}\frac{j+1}{n-j}\frac{1-a}{a}\leq \sqrt{2}n \frac{1-a}{a}\leq  \frac{\sqrt{2}n}{\frac{n^2}{2}-1}=\frac{\sqrt{2}}{\frac{n}{2}-\frac{1}{n}}<\frac{2\sqrt{2}}{3}<1,
\]
which gives that
\begin{equation*}
\begin{split}
&\sum_{j=1}^{n}j^{-\frac12}C_n^j a^j (1-a)^{n-j}\leq \frac{1}{\sqrt{n}}a^n \sum_{j=1}^{n}\Big(\frac{2\sqrt{2}}{3}\Big)^{n-j}
\leq  \frac{1}{\sqrt{n}}\frac{1}{1-\frac{2\sqrt{2}}{3}}=\frac{C}{\sqrt{n}}.
\end{split}
\end{equation*}
Similarly,
\begin{equation*}
\sum_{j=1}^{n}j^{-\frac14}C_n^j a^j (1-a)^{n-j}\leq \frac{1}{n^{\frac{1}{4}}}a^n \sum_{j=1}^{n}\Big(\frac{2\cdot \sqrt[4]{2}}{3}\Big)^{n-j}
\leq  \frac{1}{n^{\frac{1}{4}}}\frac{1}{1-\frac{2\cdot \sqrt[4]{2}}{3}}=\frac{C}{\sqrt[4]{n}}.
\end{equation*}
\end{proof}

The following lemma is an adapted version of \cite[Lemma C.1]{gonzalez2025sparse}, see also Lemma \ref{lem:C1GEN25}.

\begin{lemma}[Lemma C.1 in \cite{gonzalez2025sparse}, pointwise control of dual potentials]\label{lem:LinftyBoundSchrondinger} We have
   \begin{equation}\label{eq:potentialLInfityG}
\big\|g^{r,s}\big\|_{L^{\infty}(\mu^{B_r})}, \big\|g^{r,s}_n\big\|_{L^{\infty}(\mu^{B_r})}\leq  2C_p(r+s)^p
\end{equation}
and 
\begin{equation}\label{eq:potentialLInfityF}
\big\|f^{r,s}\big\|_{L^{\infty}(\mu^{B_r})}, \big\|f^{r,s}_n\big\|_{L^{\infty}(\mu^{B_r})}\leq t_0+7C_p(r+s)^p. 
\end{equation}
\end{lemma}
\begin{proof}
For a function~$h$ define
\begin{equation}
    \norm{h}_{\osc}:= \sup h - \inf h.
\end{equation}
According to Lemma \ref{lem:C1GEN25}, 
\begin{align}\label{eq:osc}
\begin{split}
\big\|f^{r,s}\big\|_{\osc}, \big\|g^{r,s}\big\|_{\osc}, \big\|f^{r,s}_n\big\|_{\osc}, \big\|g^{r,s}_n\big\|_{\osc}\leq 2\norm{c}_{L^{\infty}(\mu^{B_r}\otimes \nu^{B_s})}\overset{\text{Lem. }~\ref{lem: pointwise bound of cost function}}{\le} 2C_p(r+s)^p.
\end{split}
\end{align}
Since we have~$\int g^{r,s}\, d\nu^{B_r}=0$ and~$\int g^{r,s}_n d\nu_n^{B_r}=0$, we conclude by the intermediate value theorem,
\begin{equation*}
 \min g^{r,s} \leq 0 \leq \max g^{r,s},\quad \min g^{r,s}_n  \leq 0 \leq \max g^{r,s}_n,
\end{equation*}
which implies that
\begin{equation}
    \big\|g^{r,s}\big\|_{L^{\infty}(\nu^{B_s})}\leq \big\|g^{r,s}\big\|_{\osc}\leq 2C_p(r+s)^p,\quad \big\|g_n^{r,s}\big\|_{L^{\infty}(\nu_n^{B_s})}\leq \big\|g_n^{r,s}\big\|_{\osc}\leq 2C_p(r+s)^p.
\end{equation}
On the other hand, according to Lemma~\ref{lem:C1GEN25} below, %
\begin{align*}
\norm{f^{r,s}+ g^{r,s}}_{L^\infty(\mu^{B_r}\otimes \nu^{B_s})}&\le \inf  (f^{r,s}+g^{r,s}) + \big\|f^{r,s}\big\|_{\osc}+ \big\|g^{r,s}\big\|_{\osc} \\
&\leq t_0 + 5\norm{c}_{L^{\infty}(\mu^{B_r}\otimes \nu^{B_s})}.
\end{align*}
 We then have
 \begin{align*}
\norm{f^{r,s}}_{L^{\infty}(\mu^{B_r})}&\leq \norm{f^{r,s}+ g^{r,s}}_{L^\infty(\mu^{B_r}\otimes \nu^{B_s})}+\big\|g^{r,s}\big\|_{L^{\infty}(\nu^{B_s})}\\
&\le  t_0 + 7\norm{c}_{L^{\infty}(\mu^{B_r}\otimes \nu^{B_s})}\leq t_0+7C_p(r+s)^p.
 \end{align*}
 An analogous statement holds for~$f_n^{r,s}$.
\end{proof}

\begin{lemma}[Variance of dual potentials] \label{lem:varfg}
We have
\begin{align}\label{e:varfg}
     \Var_{\mu^{B_r}}(f^{r,s})&\leq (C_p(r+s)^p)^2\\
     \Var_{\nu^{B_s}}(g^{r,s})&\leq (C_p(r+s)^p)^2.
\end{align}
\end{lemma}
\begin{proof}
We only bound~$\Var_{\mu^{B_r}}(f^{r,s})$; the estimate for~$\Var_{\nu^{B_s}}(g^{r,s})$ follows analogously. For this we note that
\begin{equation}\label{e:varosc}
   \Var_{\mu^{B_r}}(f^{r,s})= \Var_{\mu^{B_r}}(f^{r,s}-c_f)\leq \norm{f^{r,s}-c_f}^2_{L^{\infty}}=\frac{1}{4}\norm{f^{r,s}}^2_{\osc} \overset{\eqref{eq:osc}}{\le } (C_p(r+s)^p)^2.
\end{equation}
where~$c_f:=\frac12(\sup f^{r,s}-\inf f^{r,s})$.
This completes the proof. 
\end{proof}

The following lemma is  \cite[Proposition 18]{stromme2023minimum}.
\begin{lemma}[Proposition 18 in \cite{stromme2023minimum}]\label{lem: average inverse mass}
    Suppose~$\rho\in \mathcal{P}(\R^d)$ has compact support. Then
    \[
    \int \rho(B_\delta(z))^{-1}\,  \rho(dz) \leq \mathcal{N}\big(\mathrm{spt}(\rho),\frac{\delta}{4}\big).
    \]
\end{lemma}

We also need the following elementary result.
\begin{lemma}\label{lem: calcfact}
For every~$a>0$ and~$x\geq a^2$ we have 
\begin{equation}\label{eq:xgeqalogx}
    a\log x\le x.
\end{equation}
\end{lemma}
\begin{proof}[Proof of Lemma~\ref{lem: calcfact}]
We distinguish the two cases~$a\in (0,e]$ and~$a>e$.   

\restartcases

\case[$a\leq e$] Observe that
\begin{equation}\label{eq:easy}
\frac{\partial}{\partial x}(x-a\log x)=1-\frac{a}{x},
\end{equation}
which implies that for~$x>0$, the function~$x \mapsto x-a\log x$ attains its minimum value when~$x=a$. The conclusion~\eqref{eq:xgeqalogx} follows from the fact that~$\log a\leq 1$.

\case[$a>e$]
As~$x\geq a^2\geq a$, we conclude from~\eqref{eq:easy} that the function~$x \mapsto x-a\log x$ is monotonically increasing. Thus
\begin{equation}
    x-a\log x\geq a^2 -2 a\log a = a(a-2\log a)>0,
\end{equation}
where the last inequality uses the fact that for any~$a>0$,~$a-2\log a>0$.
\end{proof}

\begin{lemma}\label{l:estGamma}
Let~$s>0$. If~$x\geq 4s^2\vee e$, then
 \begin{equation}
      \Gamma\big(s,(s \vee 1) \log(x)\big)\leq \frac{1}{x}.
 \end{equation}
\end{lemma}

\begin{proof}[Proof of Lemma~\ref{l:estGamma}]
We distinguish the two cases~$s<1$ and~$s\geq 1$.
\restartcases
\case[$s<1$]
Notice that~$x\geq e$ and thus~$s~$ by direct computation
\begin{equation}
     \Gamma(s,\log(x))=\int_{\log(x)}^{\infty}t^{s-1}e^{-t}\,d t\leq \int_{\log(x)}^{\infty}e^{-t}\,d t=\frac{1}{x}.
\end{equation}

\case[$s\geq 1$]
Firstly, we recall the fact that when~$x> 0$,
\begin{equation}
    \frac{\partial}{\partial x} \frac{\log x}{x} =\frac{1}{x^2}(1-\log x).
\end{equation}
This implies that the function~$\log x/x$ is non-increasing when~$x\geq e$ and non-decreasing when~$x<e$. As a result,
\begin{equation}\label{eq: max logx/x}
    \max_{x>0}\frac{\log x}{x}\leq \frac{\log e}{e}=\frac{1}{e}.
\end{equation}
Therefore, when~$x\geq 4s ^2\vee e$,
\begin{equation}\label{eq:logxoverx}
    \frac{\log(x)}{x}\leq \frac{2\log (2s)}{4s^2}=\frac{1}{s}\frac{\log (2s)}{2s} \overset{\eqref{eq: max logx/x}}{\leq} \frac{1}{s}\frac{1}{e}.
\end{equation}
According to \cite[Satz 4.4.3]{Gabcke1979},  for~$y>s$ and~$s\geq 1$,
\begin{equation*}
    \Gamma(s,y)\leq s e^{-y}y^{s-1}.
\end{equation*}
We plug in~$y=s\log (x)$
and obtain
\begin{align*}
\Gamma\big(s,s\log(x)\big) \leq s\frac{1}{x^{s}} \big(s\log(x)\big)^{s-1}
&\leq  s^{s}\Big(\frac{\log(x)}{x}\Big)^{s-1} \frac{1}{x}\\
&\overset{\eqref{eq:logxoverx}}{\leq} s^{s}\Big(\frac{1}{s}\frac{1}{e}\Big)^{s-1} \frac{1}{x}=\frac{s}{e^{s-1}}\frac{1}{x} 
\le \frac{1}{x},
\end{align*}
where the last inequality uses the fact that~$\frac{s}{e^{s-1}}\le 1$ for all~$s\in \R$.
\end{proof}

\section{Proof of Lemma~\ref{lem:duality}}

The proof is analogous to that of \cite[Proposition 2.3]{gonzalez2025sparse}.

\begin{lemma}[Lemma C.1 in \cite{gonzalez2025sparse}]\label{lem:C1GEN25}
    Let~$P, Q$ be probability measures on~$\R^d$ with supports~$\Omega,\Omega'$. Let~$c\in\mathcal{C}(\Omega\times\Omega')$ be bounded and have modulus of continuity~$\rho$. Given any bounded measurable function~$g:\Omega'\to\R$, there exists a unique function~$f:\Omega\to\R$ such that
    \begin{equation}
  \label{eq:fmargin}        \int \psi'\paren[\Big]{f(x)+g(y)-c(x,y)}\, d Q(y)=1\quad \text{for all}~x\in\Omega.
    \end{equation}
    Moreover,~$f$ is uniformly continuous with modulus~$\rho$ and its oscillation is bounded as
    \begin{equation}
    \sup_{x\in\Omega}f(x)-\inf_{x\in\Omega}f(x)\leq \norm{c}_{\infty},
    \end{equation}
    while~$\inf_{x,y}\{f(x)+g(y)\}\leq t_0+\norm{c}_{\infty}$ and~$\sup_{x,y}\{f(x)+g(y)\}\geq t_0-\norm{c}_{\infty}$, where~$t_0$ is defined in Assumption~\ref{assum:psi}. Finally,~$f$ solves the concave optimization problem
    \begin{equation}\label{eq:concaveopt}
        \sup_{f\in L^{\infty}(P)}\int \paren[\big]{f\oplus g -\psi\paren{f\oplus g-c}}\, d (P\otimes Q).
    \end{equation}
\end{lemma}
\begin{proof}
As~$g$ and~$c$ are bounded, for any fixed~$x$,~$\lim_{s\to\infty}\psi'(s+g(y)-c(x,y))=\infty$ and~$\lim_{s\to -\infty}\psi'(s+g(y)-c(x,y))<1$ by the properties of~$\psi$. As~$\psi'$ is continuous, according to the intermediate value theorem, there exists a value~$f(x)$ such that~\eqref{eq:fmargin} holds. Let~$x,\tilde{x}\in\Omega$ and assume without loss of generality that~$f(\tilde x)\leq f(x)$. As~$\psi'$ is nondecreasing,~\eqref{eq:fmargin} yields,
\begin{align}
 \int \psi'\paren[\Big]{f(x)+g(y)-c(x,y)}\, d Q(y)&=1= \int \psi'\paren[\Big]{f(\tilde{x})+g(y)-c(\tilde{x},y)}\, d Q(y)\\
 &\leq \int \psi'\paren[\Big]{f(\tilde{x})+g(y)-c(x,y)+\rho(\norm{x-\tilde x})}\, d Q(y).
\end{align}
Since~$\psi'$ is strictly increasing on~$[t_0-\delta,\infty)$, this implies that~$f(x)\leq f(\tilde{x})+\rho(\norm{x-\tilde{x}})$. Thus,~$f$ has modulus of continuity~$\rho$. Applying the same argument with~$x=\tilde{x}$ shows that~$f(x)$ is uniquely determined by~\eqref{eq:fmargin} and also that the oscillation of~$f$ is bounded by that of~$c$: 
\begin{equation}
    \sup_{x}f(x)-\inf_{x}f(x)\leq \sup_{x,y}c(x,y)-\inf_{x,y}c(x,y)\leq \norm{c}_{\infty}.
\end{equation}

Since~$\psi'(t_0)=1$ by Assumption~\ref{assum:psi}, the equation~\eqref{eq:fmargin}
implies that
\begin{equation}
    \inf_{x,y}\{f(x)+g(y)-c(x,y)\}\leq t_0\leq \sup_{x,y}\{f(x)+g(y)-c(x,y)\}.
\end{equation}
Thus~$ \inf_{x,y}\{f(x)+g(y)\}\leq t_0+\norm{c}_{\infty}$ and~$\sup_{x,y}\{f(x)+g(y)\}\geq t_0-\norm{c}_{\infty}$. Finally, since~$f$ satisfies~\eqref{eq:fmargin} which is the first order condition for optimality, we have that~$f$ solves the concave optimization~\eqref{eq:concaveopt}. 
\end{proof}

\begin{proof}[Proof of Lemma~\ref{lem:duality}]
Notice that the statement of Lemma~\ref{lem:duality} is included in \cite[Proposition 2.3, i-vi]{gonzalez2025sparse}. The proof of Lemma~\ref{lem:duality} is identical to the proof of \cite[Proposition 2.3, i-vi]{gonzalez2025sparse}, where we replace the use of \cite[Lemma C.1]{gonzalez2025sparse} by Lemma~\ref{lem:C1GEN25} above.
\end{proof}

\section{Proof of Corollary~\ref{thm:nucompact}}\label{sec:ruiyu_compact}
In this section we prove Theorem~\ref{thm:nucompact}. To simplify notation, we always assume throughout this section, that~$\mu$ satisfies Assumption~\ref{assumption:Psialpha} and that~$\nu$ is compactly supported. We begin by stating two preparing lemmas The first one is an analogue of Lemma~\ref{lem: nx>0, or ny>0}.
\begin{lemma}\label{lem:nr=0cpt} For any~$r>0$ we have
   \begin{align*}
    \MoveEqLeft\E \big[|\mathcal{C}_{\veps}(\mu_n,\nu_n)-\mathcal{C}_{\veps}(\mu,\nu)|\,\big| n_r=0\big]\\
&\leq C \big[ 1+M_p(\mu)^p+M_p(\nu)^p+ M_p(\mu_n^{B_r^c})^p + M_p(\nu_n)^p \big].
\end{align*}
\end{lemma}

\begin{proof}
Recall~\eqref{eq:Cmunu} and~\eqref{eq:Cmunuhat}. Since  we have~$\mu_n=\mu_n^{B_r^c}$ on the event~$\{n_r=0\}$, it follows from the triangle inequality that
\begin{equation}\label{eq:CmuNx0cpt}
|\mathcal{C}_{\veps}(\mu,\nu)-\mathcal{C}_{\veps}(\mu_n,\nu_n)|
\leq  C \big[ 1+M_p(\mu)^p+M_p(\nu)^p+ M_p(\mu_n^{B_r^c})^p + M_p(\nu_n)^p \big].
\end{equation}
Taking conditional expectations on both sides of~\eqref{eq:CmuNx0cpt} %
finishes the proof.
\end{proof}

\noindent The second lemma follows directly from Lemma~\ref{lem: expectation of empirical approx}.

\begin{lemma}\label{lem: expectation of empirical approx cpt}
For any~$r>0$ we have
\begin{align}
\MoveEqLeft\sum_{i=1}^{n}\E \big[ |\mathcal{C}_{\veps}(\mu_n^{B_r},\nu_n)-\mathcal{C}_{\veps}(\mu_n,\nu_n)|\,\big| n_r=i \big]\cdot \mathbb{P}(n_r=i)\\
&\leq C\Big(M_p(\mu)^p+M_p(\nu)^p\Big)^{\frac{p-1}{p}}\cdot\bigg[ \Big(M_p(\mu)^p+ M_p(\mu^{B_r^c})^p \Big)\mu(B_r^c)\bigg]^{\frac1p}.
\label{eq:Enr>0cpt}
\end{align}
\end{lemma}

\begin{proof}

\restartsteps

\step
We first claim that
if~$n_r>0$, then 
\begin{align}
\E\big[|\mathcal{C}_{\veps}(\mu_n,\nu_n)-\mathcal{C}_{\veps}(\mu_n^{B_r},\nu_n)|\,\big| n_r\big]
&\leq C\Big(M_p(\mu)^p+M_p(\nu)^p\Big)^{\frac{p-1}{p}}\\
& \qquad\cdot\bigg[M_p(\mu)^p+ M_p\big(\mu^{B_r^c}\big)^p\Big)\cdot\Big(1-\frac{n_r}{n}\Big)\bigg]^{\frac1p}. \label{eq:lem62claim1st}
\end{align}
    We now proceed to prove~\eqref{eq:Enr>0cpt} assuming the above claim, which immediately gives that
\begin{equation}\label{eq: lem62step1}
    \sum_{i=1}^{n}\E \big[ |\mathcal{C}_{\veps}(\mu_n^{B_r},\nu_n)-\mathcal{C}_{\veps}(\mu_n,\nu_n)|\,\big| n_r=i\big] \cdot \mathbb{P}(n_r=i)
        \le   C\Big(M_p(\mu)^p+M_p(\nu)^p\Big)^{\frac{p-1}{p}}\cdot I_1
\end{equation}
where
\begin{equation}
I_1= \sum_{i=1}^{n}\bigg[\Big(M_p(\mu)^p+ M_p\big(\mu^{B_r^c}\big)^p\Big)\cdot\Big(1-\frac{n_r}{n}\Big)\bigg]^{\frac1p}\cdot \mathbb{P}(n_r=i).
\end{equation}
By Jensen's inequality,
\begin{equation}
\begin{split}
I_1^p\quad&\leq  
\Big[M_p(\mu)^p+ M_p(\mu^{B_r^c})^p \Big]\sum_{i=1}^{n} \Big(1-\frac{i}{n}\Big) \mathbb{P}(n_r=i)\\
&\overset{\mathclap{\eqref{eq:E1-nr/n}}}{\leq}\quad \Big[M_p(\mu)^p+ M_p(\mu^{B_r^c})^p \Big]\mu(B_r^c).\label{eq:lem62I1}
\end{split}
\end{equation}
Plugging~\eqref{eq:lem62I1} into~\eqref{eq: lem62step1} finishes the proof.

\step We now prove~\eqref{eq:lem62claim1st} following the proof of Lemma~\ref{lem: estimate of second term-after expectation} closely.  Observe that  Lemma~\ref{lem: estimate second term-before take the expetation} and H\"older's inequality yield
\begin{equation}\label{eq:lem62claim}
\E\big[|\mathcal{C}_{\veps}(\mu_n,\nu_n)-\mathcal{C}_{\veps}(\mu_n^{B_r},\nu_n)|\,\big| n_r\big]
\leq  C(A_1)^{\frac{p-1}{p}}\cdot (A_2)^{\frac1p},
\end{equation}
where
\begin{equation*}
\begin{split}
A_1&\overset{\mathclap{\eqref{eq:defL}}}{=}\E\big[\big(M_p(\mu_n)+M_p(\nu_n)+M_p(\mu_n^{B_r})+M_p(\nu_n)\big)^p\,\big| n_r\big]\\
A_2&=\E[W_p(\mu_n,\mu_n^{B_r})\, |n_r]\overset{\eqref{eq:WpmuNmuBN}}{\leq}\E\Big[\Big(\frac{1}{n_r}-\frac{1}{n}\Big)\sum_{X_i\in B_r} |X_i|^p+\frac{1}{n} \sum_{X_i\notin B_r}|X_i|^p\,\Big|n_r\Big].
\end{split}
\end{equation*}
For~$A_1$, the inequality~\eqref{eq:munMoment} and~\eqref{eq:munBMoment} give
\begin{equation}
    A_1 \leq C\Big((M_p(\mu))^p+(M_p(\nu))^p\Big). \label{eq:lem62A1}
\end{equation}
For~$A_2$, using the inequality~\eqref{eq:EXrnr} and~\eqref{eq:EXrcnr}, we have 
\begin{equation}\label{eq:lem62A2}
    A_2\leq \Big(M_p(\mu)^p +M_p(\mu^{B_r^c})^p\Big)\cdot  \Big(1-\frac{n_r}{n}\Big).
\end{equation}
Plugging~\eqref{eq:lem62A1} and~\eqref{eq:lem62A2} into~\eqref{eq:lem62claim} finishes the proof.
\end{proof}

\noindent Now we are in a position to prove Corollary~\ref{thm:nucompact}.
\begin{proof}[Proof of Corollary~\ref{thm:nucompact}]
The proof is very similar to the one of Theorem~\ref{thm: main}, with a few simplifications.

By assumption there exists~$r^{\nu}>0$ such that~$\supp(\nu)\subseteq B_{r^{\nu}}(0)$. Let us fix~$n\geq 5$ and choose~$r=r_n^{\mu}, s =r^{\nu}$,
where~$r_n^{\mu}$ is defined in~\eqref{eq:chooseRxRy}. 
By Assumption~\ref{assumption:Psialpha}, we have
\begin{equation}\label{eq:TailChooseRScpt}
\mu((B_n^\mu)^c)\le  \frac{2}{n^p},\quad \nu(B_s^c)=0.
\end{equation}
By the tower property of conditional expectation we have
\begin{equation}\label{eq:eg1T1T2}
\E [|\mathcal{C}_{\veps}(\mu_n,\nu_n)-\mathcal{C}_{\veps}(\mu,\nu)|]
=\E\big[\E \big[|\mathcal{C}_{\veps}(\mu_n,\nu_n)-\mathcal{C}_{\veps}(\mu,\nu)|\,\big|n_r\big]\big]=T_1+T_2,
\end{equation}
where
\begin{align}
T_1&:=\E \big[|\mathcal{C}_{\veps}(\mu_n,\nu_n)-\mathcal{C}_{\veps}(\mu,\nu)|\,\big|n_r=0\big]\cdot \mathbb{P}(n_r=0),\\
T_2&:=\sum_{i}^{n}\E \big[|\mathcal{C}_{\veps}(\mu_n,\nu_n)-\mathcal{C}_{\veps}(\mu,\nu)|\,\big| n_r=i\big]\cdot \mathbb{P}(n_r=i).
\end{align}
We bound the two terms~$T_1$ and~$T_2$ separately. For the term~$T_1$, Lemma~\ref{lem:nr=0cpt} and~\eqref{eq:GammaChooseRS} implies that
\begin{equation}\label{eq:eg1T1}
T_1\overset{\eqref{eq:TailChooseRScpt}}{\leq}  C \Big( 1+c_{\mu}^{-\frac{p}{\alpha_{\mu}}}+M_p(\nu)^p\Big)\Big(\frac{2}{n^{\frac{p}{2}}}\Big)^{n-1}\leq C \Big( 1+c_{\mu}^{-\frac{p}{\alpha_{\mu}}}+M_p(\nu)^p\Big)n^{-\frac12}. 
\end{equation}
We turn to~$T_2$. By the triangle inequality we have
\begin{align}
\begin{split}
\E \big[ |\mathcal{C}_{\veps}(\mu_n,\nu_n)-\mathcal{C}_{\veps}(\mu,\nu)|\,\big| n_r\big]&\le \E\big[|\mathcal{C}_{\veps}(\mu_n,\nu_n)-\mathcal{C}_{\veps}(\mu_n^{B_r},\nu_n)|\,\big| n_r]\\
&\quad +\E \big[|\mathcal{C}_{\veps}(\mu_n^{B_r},\nu_n)-\mathcal{C}_{\veps}(\mu^{B_r},\nu)|\,\big| n_r\big]\\
&\quad +|\mathcal{C}_{\veps}(\mu^{B_r},\nu)-\mathcal{C}_{\veps}(\mu,\nu)|. \label{eq:eg1.6T2}
\end{split}
\end{align}
For the first term, 
 \begin{align}
 \begin{split}
&\sum_{i=1}^{n}\E \big[|\mathcal{C}_{\veps}(\mu_n,\nu_n)-\mathcal{C}_{\veps}(\mu_n^{B_r},\nu_n)|\,\big| n_r=i\big]\cdot \mathbb{P}(n_r=i)\\
&\qquad\overset{\mathclap{\eqref{eq:Enr>0cpt}}}{\leq} C\Big(M_p(\mu)^p+M_p(\nu)^p\Big)^{\frac{p-1}{p}}\cdot\bigg[ \Big(M_p(\mu)^p+ M_p(\mu^{B_r^c})^p \Big)\mu(B_r^c)\bigg]^{\frac1p}\\
&\qquad\overset{\mathclap{\eqref{eq:GammaChooseRS}-\eqref{eq:IntGamma1}}}{\leq} \quad \frac{C}{\sqrt{n}}\Big(c_{\mu}^{-\frac{p}{\alpha_{\mu}}}+M_p(\nu)^p\Big)^{\frac{p-1}{p}}\paren[\Big]{1+ c_{\mu}^{-\frac{1}{\alpha_{\mu}}}}. \label{eq:eg1.6T1}
\end{split}
\end{align}
We now estimate the second term on the right hand side of~\eqref{eq:eg1.6T2}. 
By Corollary~\ref{cor: first term-before taking expecation} and Lemma~\ref{lem: binomial integrals} with~$a=\mu(B_n^\mu)\ge 1-2/n^2$ recalling~\eqref{eq:TailChooseRScpt},
we conclude for~$n\geq 5$

\begin{align}
&\sum_{i=1}^{n}\E \big[|\mathcal{C}_{\veps}(\mu_n^{B_r},\nu_n)-\mathcal{C}_{\veps}(\mu^{B_r},\nu)|\,\big| n_r=i\big]\cdot \mathbb{P}(n_r=i)\\
&\overset{\mathclap{\eqref{eq:EmpApproxCpt}}}{\leq}C_p(r+s)^p \sum_{i=1}^{n}\Big(\frac{1}{\sqrt{i}}+\frac{1}{\sqrt{n}}\Big)\cdot \mathbb{P}(n_r=i)
\\
&\quad+\Bigg(1+e^{\frac{C_{\psi}\tilde{\delta}}{2(t_0/2)^{\gamma}}}
\sqrt{\mathcal{N} \Big(\supp(\nu), \frac{\tilde \delta\veps}{2C_p(r+s)^{p-1}}\Big)}\Bigg)\cdot\sum_{i=1}^{n}\bigg[\frac{4\big(t_0+9C_p(r+s)^p\big)}{(in)^{\frac14}}+\frac{\veps}{\sqrt{in}}\bigg]\cdot \mathbb{P}(n_r=i)\\
&\overset{\mathclap{\eqref{eq:binOverN}}}{\leq}C_p(r+s)^p\frac{C}{\sqrt{n}}+\Bigg(1+e^{\frac{C_{\psi}\tilde{\delta}}{2(t_0/2)^{\gamma}}}
\sqrt{\mathcal{N} \Big(\supp(\nu), \frac{\tilde \delta\veps}{2C_p(r+s)^{p-1}}\Big)}\Bigg)\bigg[\frac{C_p(r+s)^p}{\sqrt{n}}+\frac{\veps}{n}\bigg].
\end{align}
For the last term on the right hand side of~\eqref{eq:eg1.6T2}, similar as the derivation of~\eqref{eq:eg1.6T1},
\begin{equation}
\abs{\mathcal{C}_{\veps}(\mu^{B_r},\nu)-\mathcal{C}_{\veps}(\mu,\nu)}\leq \frac{C}{\sqrt{n}} \Big(c_{\mu}^{-\frac{p}{\alpha_{\mu}}}+ M_p(\nu)^p\Big)^{\frac{p-1}{p}}\paren[\Big]{1+c_{\mu}^{-\frac{1}{\alpha_{\mu}}}} .
\end{equation}
Thus, we obtain 
\begin{align}
T_2&\leq \frac{C}{\sqrt{n}} \Big(c_{\mu}^{-\frac{p}{\alpha_{\mu}}}+M_p(\nu)^p\Big)^{\frac{p-1}{p}}\paren[\Big]{1+c_{\mu}^{-\frac{1}{\alpha_{\mu}}}}
+\frac{C}{\sqrt{n}}\\
&\qquad+\Big(\frac{C(r+s)^p}{\sqrt{n}}+\frac{C\veps}{n}\Big)\cdot 
         \Bigg(1+e^{\frac{C_{\psi}\tilde{\delta}}{2(t_0/2)^{\gamma}}}
\sqrt{\mathcal{N} \Big(\supp(\nu), \frac{\tilde \delta\veps}{2C_p(r+s)^{p-1}}\Big)}\Bigg). \label{eq:eg1T2}
\end{align}
Plugging~\eqref{eq:eg1T1} and~\eqref{eq:eg1T2} into~\eqref{eq:eg1T1T2} completes the proof.
    
\end{proof}

\end{document}